\documentclass{article}

\pdfoutput=1
\PassOptionsToPackage{numbers, compress}{natbib}

\usepackage[preprint]{neurips_2022}

\usepackage{microtype}
\usepackage{graphicx}
\usepackage{booktabs} 
\usepackage{comment}

\usepackage{dsfont}
\usepackage{amsmath,amsfonts,amssymb,mathtools, amsthm}
\usepackage[noend]{algorithmic}
\usepackage[ruled,vlined,linesnumbered]{algorithm2e}
\usepackage{tikz, subcaption}

\newtheorem{theorem}{Theorem}

\newtheorem{corollary}{Corollary}
\newtheorem{lemma}{Lemma}
\newtheorem{assumption}{Assumption}

\newtheorem{proposition}{Proposition}

\newtheorem{definition}{Definition}
\newtheorem{remark}{Remark}
\newtheorem*{lemma*}{Lemma}

\newtheorem{example}{Example}

\usepackage[utf8]{inputenc} 
\usepackage[T1]{fontenc}    
\usepackage{hyperref}       
\usepackage{url}            
\usepackage{booktabs}       
\usepackage{amsfonts}       
\usepackage{nicefrac}       
\usepackage{microtype}      
\usepackage{xcolor}         

\usepackage[colorinlistoftodos,bordercolor=orange,backgroundcolor=orange!20,linecolor=orange,textsize=scriptsize]{todonotes}
\newcommand{\ilyas}[1]{\todo[inline]{\textbf{Ilyas: }#1}}

\input{preamble_neurips.tex}

\title{Sharp Analysis of Stochastic Optimization under Global Kurdyka-\L ojasiewicz Inequality}
\makeatletter
\newcommand{\printfnsymbol}[1]{%
  \textsuperscript{\@fnsymbol{#1}}%
}
\makeatother

\author{%
}

\author{\textbf{Ilyas Fatkhullin}\thanks{First two authors have equal contribution.} \\ ETH AI Center \& ETH Zurich \and \textbf{Jalal Etesami}\printfnsymbol{1} \\ EPFL\thanks{École polytechnique fédérale de Lausanne} \and  \textbf{Niao He} \\ ETH Zurich \and \textbf{Negar Kiyavash} \\ EPFL\printfnsymbol{2} }

\begin{document}

\maketitle

\begin{abstract}
We study the complexity of finding the global solution to stochastic nonconvex optimization when the objective function satisfies global Kurdyka-{\L}ojasiewicz (K{\L}) inequality and the queries from stochastic gradient oracles satisfy mild expected smoothness assumption.  We first introduce a general framework to analyze Stochastic Gradient Descent (SGD) and its associated nonlinear dynamics under the setting.  As a byproduct of our analysis, we obtain a sample complexity of  $\mathcal{O}(\epsilon^{-(4-\alpha)/\alpha})$ for SGD when the objective satisfies the so called $\alpha$-P{\L} condition, where $\alpha$ is the degree of gradient domination. Furthermore, we show that a modified SGD with variance reduction and restarting (PAGER) achieves an improved sample complexity of $\mathcal{O}(\epsilon^{-2/\alpha})$ when the objective satisfies the average smoothness assumption. This leads to the first optimal algorithm for the important case of $\alpha=1$ which appears in applications such as policy optimization in reinforcement learning. 
\end{abstract}



\section{Introduction}\label{sec:intro}
\vspace{-.2cm}
Nonconvex optimization problems are ubiquitous in machine learning domains such as training deep neural networks \citep{Goodfellow_2016} or policy optimization in reinforcement learning \citep{DeterministicPolicyGradient}. Stochastic Gradient Descent (\algname{SGD}) and its variants are driving the practical success of machine learning approaches. Naturally, understanding the limits of performance of \algname{SGD} in the nonconvex setting has become an important avenue of research in recent years \cite{ghadimi2013stochastic,arjevani2019lower,khaled2020better,nguyen2017sarah,Gower_SGD_QC_PL_21,cutkosky2019momentum,zhou2021understanding}. 

We are interested in solving the unconstrained \textit{stochastic}, \textit{nonconvex} optimization problem of the form 
\begin{equation}\label{eq:problem_online}
	\min_{x\in\mathbb{R}^d}  f(x) := \Expu{\xi \sim \cD}{f_{\xi}(x)}  ,
\end{equation}
where $f(\cdot )$ is smooth and possibly nonconvex, and $\xi$ is a random vector drawn from a distribution $\cD$. Moreover, we are interested in an important special case of \eqref{eq:problem_online}, when the expectation can be written as the average of $n$ smooth functions:
\begin{equation}\label{eq:problem_finite_sum}
	\min_{x\in\mathbb{R}^d} \Big[ f(x) := \frac{1}{n} \sum \limits_{i=1}^{n} f_{i}(x) \Big].
\end{equation}
For a general nonconvex differentiable objective $f:\mathbb{R}^d\rightarrow \mathbb{R}$, finding a global minimum of $f$ is in general intractable \citep{nemirovskij1983problem,vavasis1995complexity}. There are two common  strategies to analyze optimization methods for nonconvex functions. The first one is to scale down the requirements on the solution of interest from global optimality to some relaxed version, e.g., first-order stationary point. However, such solutions do not exclude the possibility of approaching a suboptimal local minima or a saddle point. Another approach is to study nonconvex problems with additional structural assumption in the hope of convergence to global solutions.
In this direction, several relaxations of convexity have been proposed and analyzed, for instance, star convexity, quasar-convexity, error bounds condition, restricted secant inequality, and quadratic growth \citep{Karimi_PL,gower2019sgd,Gower_SGD_QC_PL_21}. Many of the aforementioned relaxations have limited application in real-world problems. 

Recently, there has been a surge of interest in the analysis of functions satisfying the so-called Kurdyka-{\L}ojasiewicz (K{\L}) inequality \citep{bolte2007lojasiewicz,bolte2014proximal}. Of particular interest is the family of functions that satisfy global K{\L} inequality. Specifically, we say that $f(\cdot)$ satisfies \emph{(global) K{\L} inequality} if there exists some continuous function $\phi(\cdot)$ such that $||\nabla f(x)|| \geq \phi\rb{ f(x) - \inf_{x} f(x) }$ for all $x \in \R^d $. If this inequality is satisfied for $\phi(t)=\sqrt{2\mu}\ t^{1/\alpha}$, then we say that \textit{(global) $\al$-P{\L} condition} is satisfied for $f(\cdot)$. The special case of K{\L} condition, $2$-P{\L}, often referred as Polyak-{\L}ojasiewicz or P{\L} condition, was originally discovered independently in the seminal works of B. Polyak, T. Le{ž}anski and S. {\L}ojasiewicz \cite{polyak1963gradient,Lezanski_1963,lojasiewicz1959probleme,lojasiewicz1963topological}. Notably, this class of problems has found many interesting emerging applications in machine learning,  for instance, policy gradient (PG) methods in reinforcement learning \citep{Mei_SoftMax_Entropy_PG,Agarwal_TheoryPolicyGradient2020,yuan2021general}, generalized linear models \citep{Mei_nonuniform_KL}, over-parameterized neural networks \citep{Allen_Zhu_overparamNN,Zeng_DL_KL_2018}, linear quadratic regulator in optimal control \citep{LQR_discrete,LQR_continuous}, and low-rank matrix recovery \citep{Low_rank_recovery_Bi}. 

 Despite increased popularity of K{\L} and $\al$-P{\L} assumptions, the analysis of stochastic optimization under it remains limited and the majority of works focus on deterministic gradient methods. Indeed until recently, only the special case of $\al$-P{\L} with  $\alpha = 2$ was mainly addressed in the literature \cite{Karimi_PL,Gower_SGD_QC_PL_21,khaled2020better,vaswani2019fast,reddi2016stochastic}. In this paper, we study the sample complexities of stochastic optimization for the broader class of nonconvex functions with global K{\L} property.


\vspace{-.2cm}
\subsection{Related Works and Open Questions}

 \paragraph{Stochastic gradient descent.} A plethora of existing works has studied the sample complexity of \algname{SGD} and its variants for finding an $\epsilon$-stationary point of general nonconvex function $f$, that is, a point $x\in\mathbb{R}^d$ for which $\mathbb{E}[||\nabla f(\hat{x})||]\leq\epsilon$. For instance, \cite{ghadimi2013stochastic} showed that for a smooth objective (one with Lipschitz gradients) under bounded variance (BV) assumption, \algname{SGD} with properly chosen stepsizes reaches an $\epsilon$-stationary point with the sample complexity of $\mathcal{O}(\epsilon^{-4})$. Recently, \cite{khaled2020better,yuan2021general} further extended the result under a much milder \textit{expected smoothness} (ES) assumption on stochastic gradient.
While this sample complexity is known to be optimal for general nonconvex functions, a naive application of this result to the function value using $\al$-P{\L} condition would lead to a suboptimal $\mathcal{O}(\epsilon_f^{\nfr{-4}{\al}})$ sample complexity for finding an $\epsilon_f$-optimal solution, i.e., $\Exp{f(x) - f^{\star}} \leq \epsilon_f$.
Recently, \cite{fontaine2021convergence} studied \algname{SGD} and established convergence rates for $\al$-P{\L} functions under BV assumption. Their  sample complexity result is $\mathcal{O}(\epsilon_f^{\nfr{-(4- \al)}{\al}})$ in our notation. Later \citep{li2021convergence} considered \algname{SGD} scheme with random reshuffling under local and global K{\L} conditions and provided convergence in the iterates for $\al \in (1, 2]$.  We note that our proof techniques are different from \citep{fontaine2021convergence} and \citep{li2021convergence} and are not limited merely to the case of BV assumption. 
In this work, we will answer the following open question:
\begin{quote}
    \textit{What is the exact performance limit of \algname{SGD} under global K{\L} condition and a more practical model of stochastic oracle?}
\end{quote}

\vspace{-.2cm}
\paragraph{Variance reduction.} There has been extensive research on development of algorithms which improve the dependence on $n$ and/or $\epsilon$ for both problems \eqref{eq:problem_online} and \eqref{eq:problem_finite_sum} (over simple methods such as \algname{SGD} and Gradient Descent (\algname{GD}) ). One important family of techniques\footnote{Another independent direction is to make use of higher order information \citep{Nesterov_Polyak_2006,fang19a_escaping_saddles,Arjevani_2nd_Order_Stoch}.} is \textit{variance reduction}, which has emerged from the seminal works of Blatt et. al \cite{Blatt_IGM_2007}. The main idea of \textit{variance reduction} is to make use of the stochastic gradients computed at previous iterations to construct a better gradient estimator at a relatively small computational cost. Various modifications, generalizations, and improvements of the variance reduction technique appeared in subsequent work, for instance, \citep{SAG_2012,SVRG_2013,SAGA_2014} to name a few. 

\vspace{-.2cm}
\paragraph{Finite-sum case.}  A number of recently proposed algorithms such as \algname{SNVRG} \cite{SNVRG}, \algname{SARAH} \cite{SARAH}, \algname{STORM} \cite{cutkosky2019momentum}, \algname{SPIDER} \cite{SPIDER}, and \algname{PAGE}  \cite{PAGE} achieve the sample complexity $\cO\big( n +  \fr{ \sqrt{n} }{\epsilon^2} \big)$ when minimizing a general nonconvex function with finite sum structure \eqref{eq:problem_finite_sum}. This result is also known to be optimal in this setting \cite{PAGE}. \cite{SARAH_ADMM} studies \algname{SARAH} in finite sum case under local K{\L} assumption and proves convergence in the iterates. The study in \cite{SARAH_ADMM} is only asymptotic analysis and the dependence on the parameters $\kappa$ and $n$, which are important in practice for quantifying the improvement over \algname{GD} and \algname{SGD} are ignored. \cite{PGD_w_SVRG_KL} proposes an \algname{SVRG}-based algorithm for K{\L} functions and \cite{SpiderBoostMomentum,SARAH,SARAH2} study other variance reduction techniques, but they only  analyze the special case $\al = 2$. Under 2-P{\L} condition, these methods further improve to $\cO\big(\rb{ n + \kappa \sqrt{n} } \log(\fr{1}{\epsilon_f}) \big)$ sample complexity\footnote{$\kappa = \nfr{\cL}{\mu}$ is the analogue of condition number, $\cL$ is defined in Assumption~\ref{as:avg_smoothness_page}. } for finding an $\epsilon_f$-optimal solution. However, it is not clear if it is possible to provide any non-asymptotic guaranties for variance reduced methods under $\al$-P{\L} condition for any $\al \in [1, 2)$.
In our work,  we will answer the following open question: 
\begin{quote}
\textit{ What is the extent of improvement any variance reduction scheme can provide under global $\alpha$-P{\L} condition for finite-sum objectives of the form \eqref{eq:problem_finite_sum}? }
\end{quote}

\vspace{-.2cm}
\paragraph{Online/streaming case.} While variance reduction methods have been initially designed for problems of the form \eqref{eq:problem_finite_sum}, it was later discovered that they also improve over \algname{SGD} when solving \eqref{eq:problem_online} \cite{SCSG, SVRG_plus}\footnote{Under additional assumptions such as smoothness of individual functions $f_{\xi}(\cdot)$ or even milder condition such as \textit{average} $L$-\textit{smoothness} (Assumption~\ref{as:avg_smoothness_page}).}. The analysis of these methods was obtained for \textit{general nonconvex} functions (for minimizing the norm of the gradient, $\Exp{\norm{\nabla f(\hat{x})}} \leq \epsilon$) and later extended to \textit{$2$-P{\L} objectives} for minimizing the function value, $\Exp{f(x) - f^{\star}} \leq \epsilon_f$. For example, the methods in \citep{SNVRG,SARAH,SPIDER,PAGE} achieve $\cO(\epsilon^{-3})$ complexity improving over $\cO(\epsilon^{-4})$ complexity of \algname{SGD} for finding an $\epsilon$-stationary point. Under the $2$-P{\L} condition, these results can be extended to global convergence with $\cO(\epsilon_f^{-1})$ sample complexity \citep{PAGE}. However, in contrast to a general nonconvex case, variance reduction under $2$-P{\L} assumption does not show any improvement over \algname{SGD} in terms of $\epsilon_f$. We highlight that all existing analysis of variance reduction under $\al$-P{\L} condition \textit{is restricted only to a special case $\al = 2$}. We refer the reader to Appendix~\ref{sec:variace_reduction_appendix}, where we elaborate on the key difficulties in the analysis for the cases $\al \in [1,2)$. Since the direct analysis for $\al \in [1,2)$ is challenging, in order to obtain the global convergence in this setting, one could naively translate the complexity for finding a stationary point of a general nonconvex function (which is $\cO (\epsilon^{-3})$) to convergence in a function value by using $\al$-P{\L} condition: $ \sqrt{2\mu} \rb{ f(\hat{x}) - f^{\star} }^{\nfr{1}{\al}} \leq  ||\nabla f(\hat{x})||$. This would result in $\cO \big( \epsilon_f^{\nfr{-3}{\al}} \big)$ sample complexity. However, there are two serious issues with this approach. First, this complexity does not provide any improvement over \algname{SGD} in the most interesting practical case  $\al = 1$ and gives strictly worse result for all $\al > 1$. Second, the guarantees for general nonconvex optimization hold on average, in the sense that the point $\hat{x}$ is sampled uniformly from all the iterates of the algorithm. It would be more desirable to instead derive last iterate convergence guarantees under K{\L} ($\al$-P{\L}) condition.  
In this work, we will address the following open question: 
\begin{quote}
\textit{ Is it possible to accelerate the $ \cO\big( \epsilon_f^{\nfr{-(4-\al)}{\al}} \big)$ sample complexity of \algname{SGD} under global $\alpha$-P{\L} condition for stochastic objectives of the form \eqref{eq:problem_online}? }
\end{quote}

\vspace{-.2cm}
\subsection{Contributions} In this work, we provide an extensive analysis of stochastic optimization under global K{\L} condition and answer all the above questions. More precisely, our contributions are as follows
\begin{itemize}
    \item We provide a new framework for the analysis of the dynamics of \algname{SGD} under global K{\L} condition (see Section \ref{sec:main}). It is based on the analysis of \algname{SGD} dynamic which is governed by a recursive inequality (see Equation \eqref{eq:main_fixed}). As a result of this analysis, we introduce a set of conditions (see Theorem \ref{lemma_man}) for designing proper stepsizes to guarantee convergence.
    
    \item Using this framework, we provide sharp analysis of \algname{SGD} under a general ES assumption (Assumption \ref{ass:k_ES}) and demonstrate that the sample complexity $\small{\cO\big( \epsilon_f^{\nfr{-(4-\al)}{\al}} \big)}$ is tight for the dynamical system describing \algname{SGD}. 


    \item Next, we propose \algname{PAGER}, a new variance reduction scheme with parameter restart. A carefully chosen sequence of parameters of \algname{PAGER} allows the algorithm to adapt to the nonconvex geometry of the problem and establish state-of-the-art convergence guarantees for minimizing $\al$-P{\L} functions. In online setting \eqref{eq:problem_online}, we obtain $\cO\rb{ \epsilon_f^{\nfr{-2}{\al}} }$ sample complexity of \algname{PAGER}, which beats $\cO\rb{ \epsilon_f^{\nfr{-(4-\al)}{\al}} }$ complexity of \algname{SGD} for the whole spectrum of parameters $\al\in [1,2)$. In particular, for the important special case of $1$-P{\L}, this leads to the first optimal algorithm with $O(\epsilon_f^{-2})$ sample complexity, which already matches with the lower bound known for stochastic convex optimization~\cite{nemirovskij1983problem}.  

    \item Furthermore, we obtain faster rates with \algname{PAGER} in finite sum case \eqref{eq:problem_finite_sum}, providing the first acceleration over \algname{GD} and \algname{SGD} under $\al$-P{\L} condition.  

\end{itemize}
In Table~\ref{tbl:complexity_f}, we summarize the sample complexity results for stochastic optimization under $\al$-P{\L} and BV assumptions. We also establish sharp convergence results for convergence in the iterates to the set $X^{\star}$ of optimal points and provide a summary in Table~\ref{tbl:complexity_x} in the Appendix. 

\begin{table}[h]
	\caption{Summary of sample complexity results for $\al$-P{\L} functions (Assumption~\ref{as:lojasiewicz}) under average $\cL$-smoothness (Assumptions~\ref{as:avg_smoothness_page}) and bounded variance (Assumptions~\ref{as:UBV}). Quantities: $\al$ = P{\L} power; $\mu$ = P{\L} constant; $\kappa = \nfr{\cL}{\mu}$; $\sigma^2$ = variance. The entries of the table show the expected number of stochastic gradient calls to achieve  $\Exp{f(x_k) - f^\star} \leq \epsilon_f$ . \vspace{.1cm} }
	\label{tbl:complexity_f}
	\footnotesize
	\centering
	\begin{tabular}{|c|c|c|c|}
		\hline
		\bf Method & \bf Finite sum case & \bf Online case \\
		\hline
		\begin{tabular}{c} \algname{GD} \end{tabular}  & $\cO\rb{n \kappa \rb{\fr{1}{\epsilon_f}}^{\fr{2-\al}{\al}} } $ & N/A \\
		\hline
		\begin{tabular}{c}\algname{SGD} \end{tabular}  & $\cO\rb{ \fr{\kappa\sigma^2}{\mu} \rb{\fr{1}{\epsilon_f}}^{\fr{4-\al}{\al}}}$ & $  \cO\rb{\fr{\kappa\sigma^2}{\mu} \rb{\fr{1}{\epsilon_f}}^{\fr{4-\al}{\al}}}$ \\
		\hline
		\rowcolor{LightCyan}
		\begin{tabular}{c}
		 \algname{PAGER} \end{tabular}  & $ \cwO\rb{ n + \sqrt{n}\kappa  \rb{\fr{1}{\epsilon_f}}^{\fr{2-\al}{\al}} }$ (new) &  $ \cO\rb{ \rb{ \fr{\sigma^2}{\mu} +  \kappa^2  } \rb{\fr{1}{\epsilon_f}}^{\fr{2}{\al}} }$  (new) \\
		\hline
	\end{tabular}
\end{table}

\section{Assumptions and Discussion}\label{sec:ass}
In this section, we introduce the assumptions we make throughout the paper.
\begin{assumption}\label{ass:l-smooth}
The gradient of $f(\cdot)$ is Lipschitz continuous, that is, for all $x$ and $y$, $\norm{\nabla f(x)-\nabla f(y)}\leq L\norm{x-y}$,
$L>0$ is referred to as the Lipschitz constant. 
\end{assumption}

Furthermore, we assume that the objective function $f$ is lower bounded, i.e., $f^* := \inf_{x} f(x)>-\infty$, and it satisfies the following inequality

\begin{assumption}[global K{\L} or global Kurdyka-{\L}ojasiewicz]\label{ass:kl_our}
Let $\phi: \R^+ \rightarrow \mathbb{R}^+$ be a continuous function such that $\phi(0)=0$ and $\phi^2(\cdot)$ is convex. The function $f(\cdot)$ is said to satisfy global Kurdyka-{\L}ojasiewicz inequality if 
\begin{align}\label{eq:ass-kl}
       ||\nabla f(x)|| \geq \phi\rb{ f(x) - f^* } \quad \text{for all } x \in \R^d. 
\end{align}
\end{assumption}

\begin{assumption}[$\al$-P{\L} or Polyak-{\L}ojasiewicz] \label{as:lojasiewicz} 
There exists $\al \in[1,2]$ and $\mu>0$ such that 
\begin{eqnarray}
    \norm{\nabla f(x)}^\al \geq (2 \mu)^{\nfr{\al}{2}}\rb{f(x) - f^* } \quad \text{for all } x \in \R^d.
\end{eqnarray}
We refer to $\al$ as the P{\L} power and $\mu$ as the P{\L} constant.
\end{assumption}
It is straightforward to see that the $\al$-P{\L} is a special case of K{\L} with $\phi(t)=\sqrt{2\mu}\ t^{1/\alpha}$.

\paragraph{Connections with other assumptions.} Another commonly adopted way to define the global K{\L} property is to assume that $\rho^{\prime}\rb{f(x)-f^{\star} } \cdot \norm{\nabla f(x) } \geq 1$ for all $x\in \R^d$, where $\rho(t)$ is called a disingularizing function and $\rho^{\prime}(\cdot)$ denotes its derivative. Moreover, disingularizing function satisfies the following conditions, it is continuous, concave, $\rho(0)=0$, and $\rho^\prime(t)>0$ \cite{li2021convergence}.

If the above assumption holds for $\rho(t) := \fr{1}{\theta} t^{\theta}$ and $\theta > 0$, then Assumption~\ref{as:lojasiewicz} is satisfied with P{\L} power $\al = \fr{1}{1 - \theta}$. However, $\al$-P{\L} condition is more general since it allows to consider the case $\al = 1$. For example, consider the function of one variable $f(x)  = (e^x + e^{-x})/2 - 1$, then $| f^{\prime}(x) | \geq  f(x)$ for all $x$. Thus $f(x)$ satisfies Assumption~\ref{as:lojasiewicz} with $\al=1$ and $\mu = \nfr{1}{2}$. Moreover, this function is convex, but it does not satisfy inequality $\rho^{\prime}\rb{f(x)-f^{\star} } \cdot \norm{\nabla f(x) } \geq 1$ for any choice of $\rho(t)$. 

We also provide several non-convex problems for which $\al$-P{\L} holds with $\al\in[1,2]$ in the Appendix~\ref{sec:examples_appendix}. Other forms of $\phi(t)$ also appear in practice, e.g., squared cross entropy loss function satisfies the KL condition with $\phi(t)=\min\{t,\sqrt{t}\}$, \cite{scaman2022convergence}.

The intuition behind the special case $\al = 1$ is that the function is allowed to be flat near the set of optimal points $X^{\star} = \argmin_x f(x)$. 

\begin{assumption}[$k$-ES, Expected Smoothness of order $k$]\label{ass:k_ES}
The stochastic gradient estimator $g_k(x,\xi)$ is an unbiased estimate of the gradient $\nabla f(x)$ at any given point $x$ and its second moment satisfies
\begin{align}\label{eq:ecpected_smoothness}
  \Exp{ \sqnorm{g_k(x,\xi) } } \leq 2A\cdot h\big(f(x) - f^*\big) + B\cdot||\nabla f(x)||^2 + \fr{C}{b_k},\quad \text{for all }  x \in \R^d,
\end{align}
where  $A, B, C $ are non-negative constants. $h:\mathbb{R}^+\rightarrow\mathbb{R}^+$ is a concave continuously differentiable function with $h'(t)\geq0$, $h(0)=0$. The expectation is taken over random vector $\xi \sim \cD$. We call $b_k$ the cost of such estimator.
\end{assumption}

This assumption encompasses previous assumptions in the literature. For instance,  it is straightforward to see that an estimator satisfies the standard \textit{bounded variance assumption} \cite{ghadimi2013stochastic} when $h(t)=0, B=1$, and $b_k=1$ in \eqref{eq:ecpected_smoothness}. 
Gradient estimators with \textit{relaxed growth assumption} \cite{bertsekas2000gradient,bottou2018optimization} are also special cases of \eqref{eq:ecpected_smoothness} for $h(t)=0$ and $b_k=1$. 
A closely related assumption to the relaxed growth was introduced in \citep{vaswani2019fast} which holds when $h(t)=0$ and $C=0$ in \eqref{eq:ecpected_smoothness}.
\textit{Expected smoothness assumption} \citep{khaled2020better,gower2019sgd,yuan2021general} is the closest assumption to $k$-ES and it holds when $h(t)=t$ and $b_k=1$.  Notably, ES assumption is satisfied in practical scenarios such as mini-batching, importance sampling and compressed communication \cite{khaled2020better}.  More recently, it has been shown that PG method with softmax policies and log barrier regularization can be modeled using ES assumption \citep{yuan2021general}. Note that due to the first term in \eqref{eq:ecpected_smoothness}, i.e., $2A\cdot h(f(x)-f^*)$, the second moment can be large when the objective gap at $x$ is large. Such property is not captured by standard \textit{bounded variance} (Assumption~\ref{as:UBV}). The flexibility and advantages of introducing such additional term are elaborated in detail in the literature \cite{gower2019sgd,khaled2020better,grimmer2019convergence,yuan2021general}. 

We highlight two special cases for the sequence $b_k$: $b_k = \Theta( k^\tau) $ with $\tau \geq 0$ and $b_k = \Theta( q^k)$ with $q > 1$. 
For example, Monte Carlo sampling and mini-batching allow us to design estimators with such $\cb{b_k}_{k\geq 0}$ sequences.  When a gradient estimator satisfies Assumption \ref{ass:k_ES}, unless $b_k$ is bounded for all $k$, it essentially means that we have a mechanism to reduce its variance.  
More precisely, the variance decreases according to the sequence $1/b_k$. 
As we show in Section \ref{sec:variace_reduction}, if such estimator exists, it results in a better convergence rate compared to vanilla SGD and the improvement is captured by sequence $b_k$. On the other hand, usually the access to such estimator comes with a \textit{cost} proportional to $b_k$, e.g., mini-batch setting described in Section~\ref{sec:variace_reduction}. Thus we refer to $b_k$ as the \textit{cost} of the estimator $g_k(x,\xi)$.

\section{Stochastic Gradient Method}\label{sec:main}
\vspace{-.2cm}
Algorithm \ref{alg:sgd} summarizes the steps of a slightly modified \algname{SGD} which we analyze in this work. We call this algorithm \algname{SGD} with restarts.\footnote{Note that if we set $K = 1$, then Algorithm~\ref{alg:sgd} reduces to \algname{SGD} with constant step-size. }
This algorithm updates the point $x$ for $T$ number of iterations within an inner-loop. Note that, in the inner-loop, the step-size remains unchanged and the iterates are updated via $x_{t+1} = x_t - \eta g_k(x_t,\xi_t)$, where $\eta$ is the step-size and $\cb{\xi_t}_{t\geq0}$ are independent random vectors. The cost $b_k$ of the gradient estimator $g_k(x, \xi)$ remain the same within the inner loop of Algorithm~\ref{alg:sgd}.
\vspace{-.5cm}
\begin{algorithm}[h]
        \caption{\algname{SGD} with restarts}\label{alg:sgd}
        \begin{algorithmic}[1]
            \STATE Initialization: $x, T, K, \{\eta_k : k=0,...,K-1\}$
            \FOR{$k=0, \ldots, K -  1$}
		\STATE $\eta \leftarrow \eta_k$
		\FOR{$t=0, \ldots, T - 1 $}
	    \STATE 	$x \leftarrow x - \eta g_k(x, \xi_t)$ 
		\ENDFOR
		\ENDFOR
		\RETURN $x$
        \end{algorithmic}
\end{algorithm}


\subsection{Dynamics of SGD}\label{sec:g_derivation}
\vspace{-.2cm}
Let $\{{x}_t\}_{t\geq 0}$ be the sequence of points generated by the inner loop of Algorithm~\ref{alg:sgd}, and Assumptions \ref{ass:l-smooth}, \ref{ass:kl_our}, \ref{ass:k_ES} are satisfied. Then the dynamics of \algname{SGD} in the inner-loop of Algorithm~\ref{alg:sgd} is characterized by 
\begin{lemma}\label{lemma:dynamic_inner}
Under Assumptions \ref{ass:l-smooth}, \ref{ass:kl_our}, and \ref{ass:k_ES} with constant cost, i.e., $b := b_k$,  we obtain
\begin{align}\label{eq:main_fixed}
         \delta_{t+1}\leq \delta_t+a\eta^2\cdot h\big(\delta_t\big)-\frac{\eta}{2}\phi^2(\delta_t)+ \fr{d\eta^2}{b},
\end{align}
where $\delta_t := \Exp{f(x_t) - f^{\star}}$, $a := L A$, $d := \fr{L C}{2 }$, $\eta := \eta_k$.
\end{lemma}


Understanding the dynamics of this recursion, allows us to establish the global convergence of \algname{SGD}. 
Our approach consists of two main steps: i) Finding the stationary\footnote{Stationary point of a dynamic is its convergence point.} point of \eqref{eq:main_fixed} when the inequality is replaced by equality and for a fixed step-size, $\eta_k=\eta$, which we denote by $r(\eta)$. ii) Selecting the step-sizes $\{\eta_k\}$ and sequence $\{b_k\}_{k\geq 0}$, such that the corresponding stationary points  $\{r(\eta_k)\}_{k\geq 0}$ (defined below) converge to zero as $k$ increases. 
 
The stationary point of \eqref{eq:main_fixed} after replacing inequality with equality  must satisfy the equation:
\begin{align}\label{eq:statioanry}
         a\eta^2 h(t) + \fr{d\eta^2}{b} = \frac{\eta}{2}\phi^2(t).
\end{align}
Let us call this stationary point $r(\eta)$. To complete the first step, we approximate  $r(\eta)$  by a polynomial function of $\eta$. 
In other words, we find $\nu\in\mathbb{R}^+$ such that $r(\eta)=\Theta(\eta^{\nu})$. 

For the second step of our framework, we should design the stepsizes. 
Next result introduces a set of conditions that allow us to design the stepsizes, which will guarantee convergence. The detailed derivations are presented in the Appendix. 


\begin{theorem}\label{lemma_man}
Suppose there exist $\nu\geq0, \{\omega_j\}_{j\geq0}$, and $\zeta\geq0$ such that $\eta_k=\Theta(k^{-\zeta})$, $r(\eta_k)=\Theta(k^{-\zeta\nu})$, $|1-\omega_k|<1$, and
\begin{align}\label{eq:condition1}
    &1+a\eta^2_k h'\big(r(\eta_k)\big)-\eta_k\phi'\big(r(\eta_k)\big)\phi\big(r(\eta_k)\big)=1-\omega_k k^{-1}.
\end{align}
Then, $\delta_k\!\!=\!\!\mathcal{O}(k^{-\zeta\nu})$ and the iteration complexity of  Algorithm~\ref{alg:sgd} with $T\!\!\!=\!\!\Omega(1/\!\min_j \omega_j)$ is $\mathcal{O}(\!\epsilon_f^{-1/(\zeta\nu)})$. 
\end{theorem}

As a consequence of Theorem \ref{lemma_man}, we present the iteration complexity of \algname{SGD} for $\al$-P{\L} functions.


\begin{corollary}\label{corr:main1}
Consider a special case of Assumption~\ref{ass:k_ES} with $h(t)=t^\beta$ and $b_k=k^{\tau}$, where $\beta\in(0,1]$ and $\tau\geq0$. 
Suppose the objective function $f$ satisfies Assumptions \ref{ass:l-smooth} and \ref{as:lojasiewicz}.
Let $\gamma:=\alpha\beta$.
Then, for any $\epsilon_f > 0$, Algorithm \ref{alg:sgd} returns a point $x$ with $\Exp{f(x) - f^{\star}} \leq \epsilon_f$ after $N := K \cdot T $ iterations.

i) If $\gamma=2$ ($\alpha=2$ and $\beta=1$), we have
\begin{align*}
   & N = \mathcal{O}(\epsilon_f^{-\frac{1}{1+\tau}}), \ \text{with}\ \eta_k=\Theta(k^{-1}).
\end{align*}

ii) If $\gamma<2$, we have
\begin{align*}
    & N = \mathcal{O}\big(\epsilon_f^{-\frac{4-\alpha}{\alpha(\tau+1)}}\big) \ \text{with}\ \eta_k=\Theta(k^{-\frac{\tau+1}{2-\alpha/2}+\tau})\ \text{if}\  \tau\leq \frac{\gamma}{4-\alpha-\gamma}, and \\
    & N = \mathcal{O}\big(\epsilon_f^{-\frac{4-\alpha-\gamma}{\alpha}}\big) \ \text{with}\ \eta_k=\Theta(k^{-\frac{2-\gamma}{4-\alpha-\gamma}})\ \text{if}\ \tau> \frac{\gamma}{4-\alpha-\gamma}.
\end{align*}
\end{corollary}
\vspace{-.2cm}
To verify the above result empirically, we simulated $\delta_t$ in \eqref{eq:main_fixed} throughout all iterations of Algorithms~\ref{alg:sgd} for different sets of parameters and presented the results in Figure \ref{fig:verify1} along with their corresponding convergence rates given in Corollary \ref{corr:main1}. 
As it is shown in these figure, the above convergence rates correctly capture the behaviour of the dynamics in \eqref{eq:main_fixed}. As an example, in Figure \ref{fig:verify1}(b), the red solid curve shows the rate of $\delta_k$, i.e., $\log(\delta_k)$ as a function of $\log(k)$ for $\gamma=1.1, \alpha=1.4$, and $\tau=0.9$. Based on Corollary \ref{corr:main1}, the convergence rate of Algorithm \ref{alg:sgd} for this setting is 
$\small{\mathcal{O}(\epsilon_f^{-0.98})}$
or equivalently $\log(\delta_k)=(-\frac{\alpha(\tau+1)}{4-\alpha})\log(k)\approx -1.02\log(k)$ which is shown by red dashed line. 
Next, we discuss how the results in Corollary \ref{corr:main1} generalizes the existing work in the literature. 
\vspace{-.2cm}
\paragraph{Comparison to related works.} Authors in \cite{khaled2020better} studied the convergence of \algname{SGD} for $2$-P{\L} objectives, under a stronger assumption than Assumption~\ref{ass:k_ES}. 
 More precisely, they assumed an estimator that satisfies Assumption \ref{ass:k_ES} with $\tau=0$ and $h(t)=t$ and obtained the convergence rate of $\mathcal{O}(\epsilon_f^{-1})$. 
 This is consistent with our rate presented in the Corollary \ref{corr:main1}. 
 It is worth noting that in this setting, $\mathcal{O}(\epsilon_f^{-1})$ is optimal \cite{khaled2020better}. 
 
The authors in \cite{yuan2021general} studied the performance of \algname{SGD} for $1$-P{\L} objectives. Assuming that the gradient estimator satisfies Assumption \ref{ass:k_ES} with $\tau=0$ and $h(t)=t$, they obtain $\mathcal{O}(\epsilon_f^{-3})$ sample complexity. This result can be recovered from Corollary~\ref{corr:main1} by setting $\tau=0, \gamma=1$, and $\alpha=1$. Note that in this case, the cost of each iteration is $b_k = 1$, which means that the iteration complexity coincides with the sample complexity. We note that our proof technique is different than in \cite{yuan2021general} and allows to consider more general assumptions. Finally, \cite{fontaine2021convergence} studied \algname{SGD} with $\al$-P{\L} objectives for $\alpha\in[1,2]$ under bounded variance Assumption~\ref{as:UBV} with $b_k=1$ and obtained similar convergence rate to ours. We recover their result as a special case by setting $A=0$ and $\tau=0$ in Corollary~\ref{corr:main1}, if we set $T=1$ we also recover the same (up to a constant) step-sizes $\eta_k = \Theta \rb{ k^{-\fr{2}{4-\al}} }$. However, we highlight that our proof technique is different from \cite{fontaine2021convergence}, and more generic since it holds for a general Assumption~\ref{ass:k_ES}.


\begin{figure}[t]
        \centering
        \begin{subfigure}{0.33\textwidth}
            \centering
            \includegraphics[width=\textwidth]{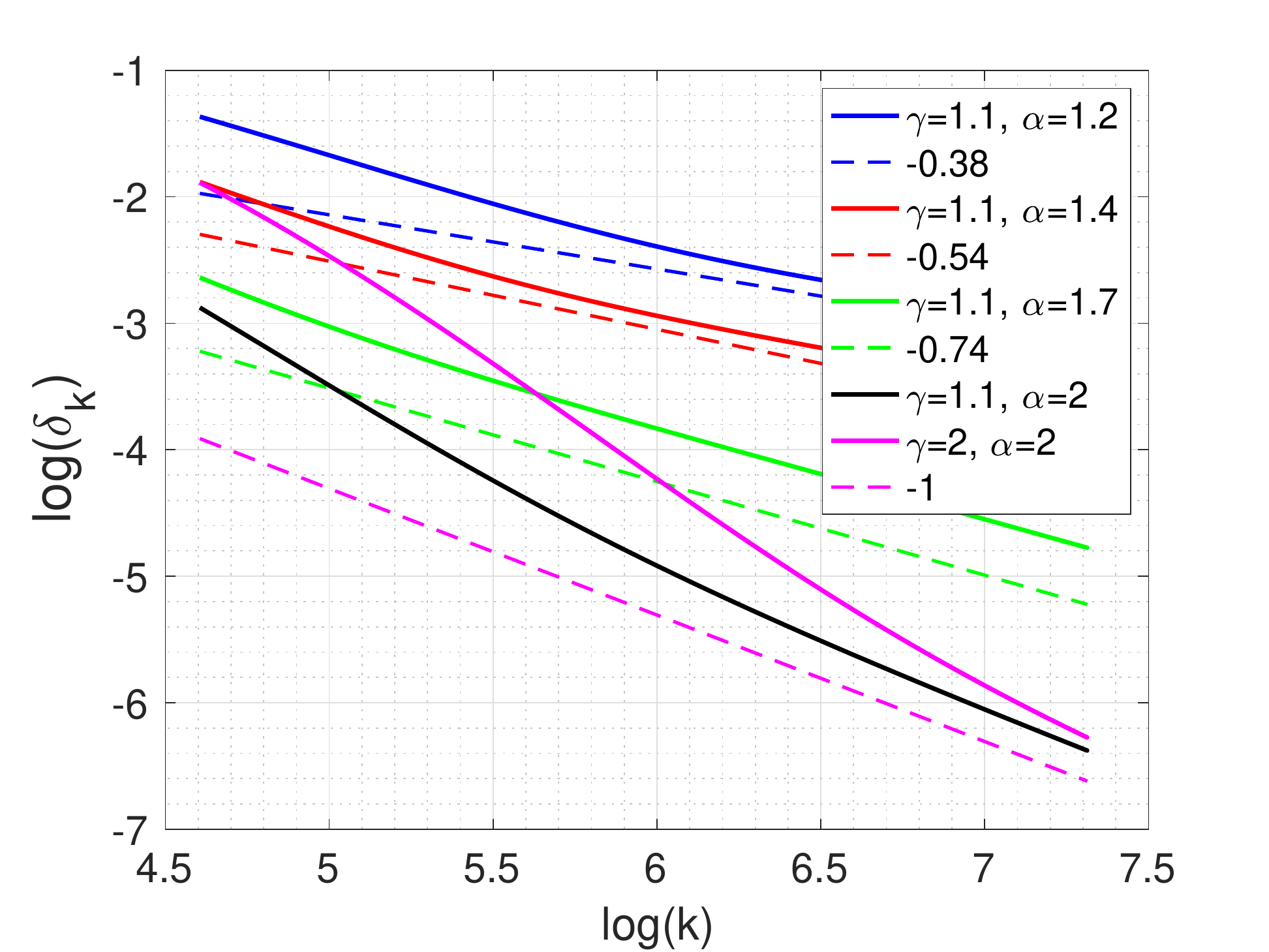}
            \caption{$\tau=0$.}
            \label{fig: tau0}
        \end{subfigure}
        \begin{subfigure}{0.33\textwidth}
            \centering
            \includegraphics[width=\textwidth]{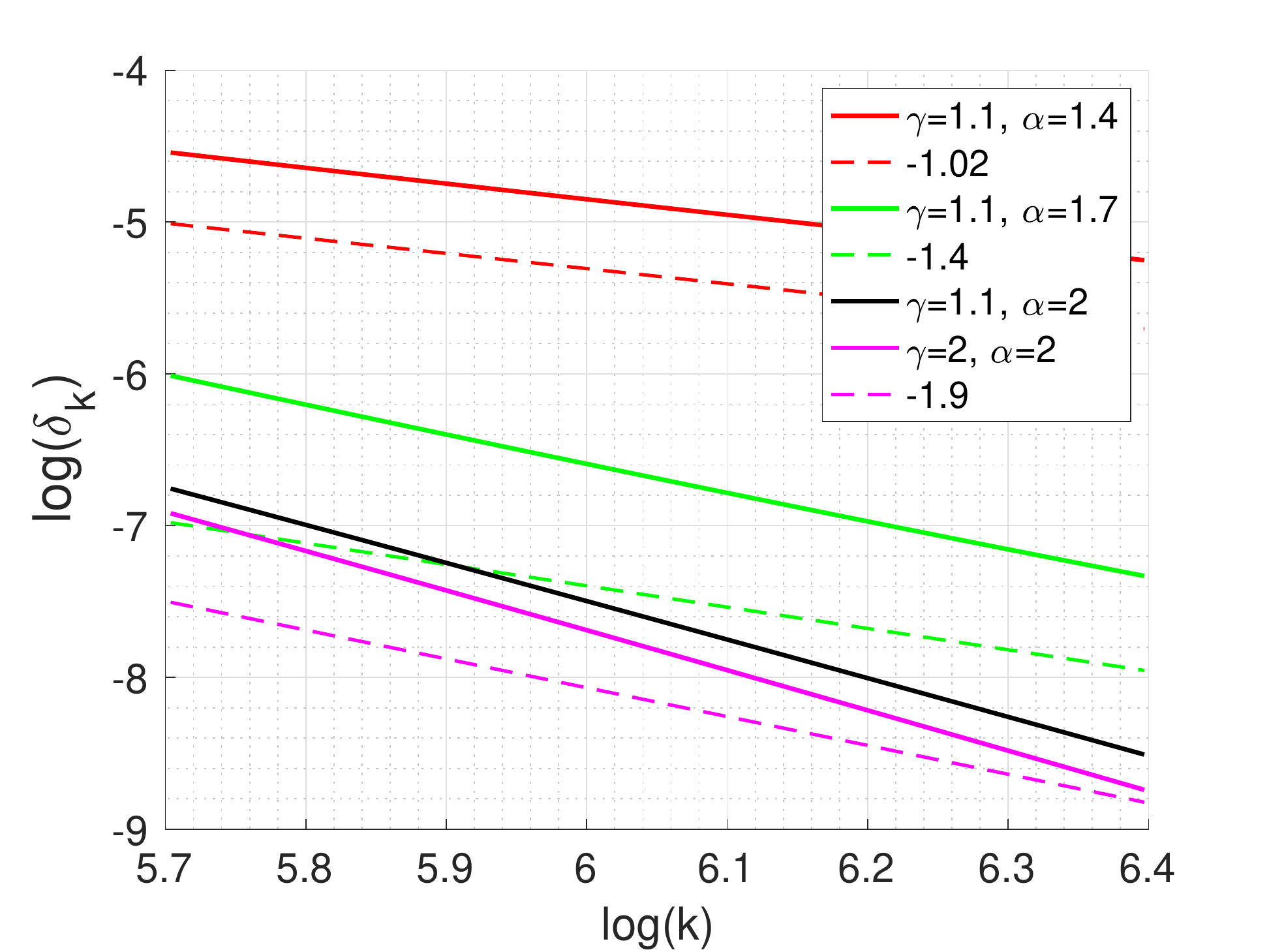}
            \caption{$\tau=0.9$.}
            \label{fig: tau9}
        \end{subfigure}%
        \begin{subfigure}{0.33\textwidth}
            \centering
            \includegraphics[width=\textwidth]{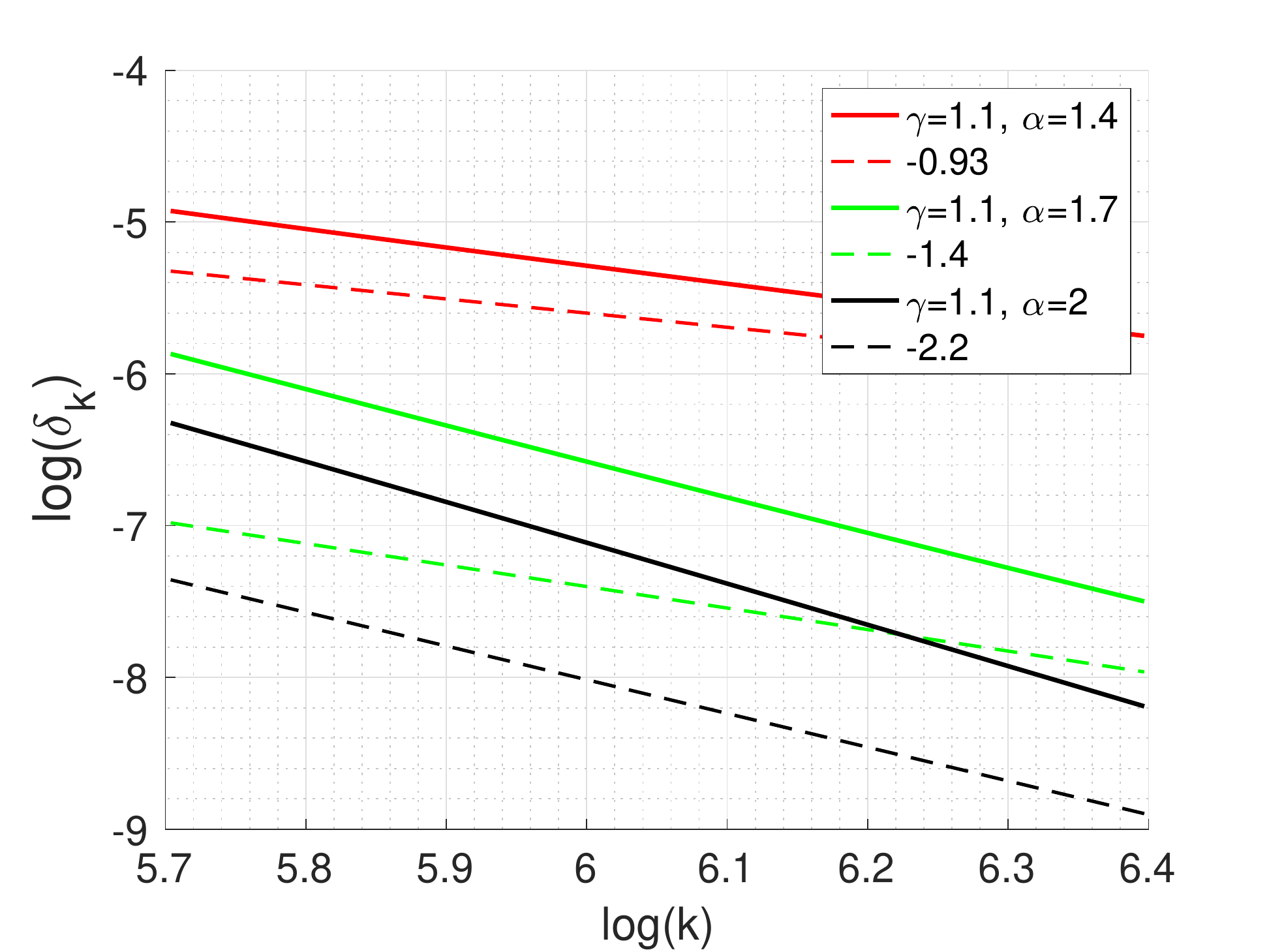}
            \caption{$\tau=2$.}
            \label{fig: tau2}
        \end{subfigure}
        \caption{Behavior of the dynamics in \eqref{eq:main_fixed} for $h(t)=t^\beta$, $\phi(t)=\sqrt{2\mu}\ t^{1/\alpha}$, $\tau\in\{0,0.9,2\}$, and different $\alpha,\beta$. Each solid line shows $\log(\delta_k)$ as a function of $\log(k)$, for a given set of parameters and each dashed line shows the corresponding theoretical convergence rate of $\delta_k$ presented in Corollary \ref{corr:main1}. The numbers assigned to dashed lines indicate the slope of those lines.  (a) and (b) verify the case corresponding to $\tau\leq\gamma/(4-\alpha-\gamma)$ and (c) verifies the case $\tau>\gamma/(4-\alpha-\gamma)$. Note that the distance between the dashed and solid lines is due to constant factors.}\label{fig:verify1}
    \end{figure}

\subsection{Sample complexity of SGD}\label{subsec:minibatch}
The result of Corollary~\ref{corr:main1} suggests that by increasing the cost of the gradient estimator $b_k$ over the iterations, one can achieve a better iteration complexity of Algortihm~\ref{alg:sgd}. In particular, it improves with $\tau$ until it reaches the minimum $N = \cO\rb{\epsilon_f^{-\frac{4-\alpha-\gamma}{\alpha}}}$ at $\tau = \frac{\gamma}{4-\alpha-\gamma}$ and does not change for larger values of $\tau$. However, we are merely interested in the iteration complexity in practice, since the computational cost at each iteration can be prohibitively large. A more adequate measure is the total computational cost (sample complexity) of the method. It is interesting whether increasing $b_k$ over the iterations may also result in a better sample complexity for finding an $\epsilon_f$-optimal solution, than for the constant choice, e.g., $b_k = 1$. The following lemma shows the contrary.
\begin{proposition}\label{lemma_mini-batchsize}
Let the assumptions of Corollary \ref{corr:main1} hold, $b_k = \Theta\rb{k^{\tau}}$, $T = \Theta\rb{1}$. Then the expected total computational cost (sample complexity) of Algorithm \ref{alg:sgd} is
	 \begin{equation}
	 \text{cost} := T \cdot \sum_{k=0}^{K-1} b_k = \begin{cases} \cO\rb{ \epsilon_f^{-\fr{4-\al}{\al} }} \qquad & \text{for } 0 \leq \tau \leq \frac{\gamma}{4-\alpha-\gamma}, \notag \\
	 \cO\rb{\epsilon_f^{-\frac{\rb{4-\alpha - \gamma} \rb{\tau+1} }{\alpha }} }  \qquad & \text{for } \tau > \frac{\gamma}{4-\alpha-\gamma} . 
	 \end{cases} 
	\end{equation}
\end{proposition}
The above result implies that increasing the cost of the gradient estimator with iterations does not improve the total sample complexity of Algorithm~\ref{alg:sgd}. Therefore, one can simply select $b_k = 1$ ($\tau = 0$) and obtain $\cO\rb{ \epsilon_f^{-\fr{4-\al}{\al} }}$ sample complexity.

\vspace{-.2cm}
\subsection{Tightness of rates in Corollary \ref{corr:main1}}
\vspace{-.2cm}
In this section, we show that when $\tau=0$, the convergence rates presented in Corollary \ref{corr:main1} are tight for the dynamic \eqref{eq:main_fixed} describing the progress of \algname{SGD}. More precisely, if there exists a function $f$ and a gradient estimator satisfying the assumptions in Corollary \ref{corr:main1} such that its corresponding recursive inequality \eqref{eq:main_fixed} is an equality, then its convergence rate, presented in Corollary \ref{corr:main1} cannot be improved by any choices of stepsizes $\cb{\eta_k}_{k\geq 0}$.  
Next proposition summarizes our results about the tightness of our convergence rates in Corollary \ref{corr:main1}.


\begin{proposition}\label{pro:lower}
Consider the following recursion 
\begin{align*}
\delta_{k+1}= \delta_k+a\eta_k^2\cdot h\big(\delta_k\big)-\frac{\eta_k}{2}\phi^2(\delta_k)+ \fr{d\eta_k^2}{b_k},\quad \text{for all } k\geq 0,
\end{align*}
where $a \geq 0$, $d>0$, $h(t)=t^\beta$ with $\beta\in(0,1]$, $\phi(t)=\sqrt{2\mu} t^{1/\alpha}$ with $\alpha\in[1,2]$, and $b_k=\Theta(1)$. 
Then $\delta_k = {\Omega}(k^{-\frac{\alpha}{4-\alpha}})$ for any sequence of $\cb{\eta_k}_{k\geq 0}$. Moreover, this rate is achieved by the choice $\eta_k=\Theta(k^{-\frac{1}{2-\alpha/2}})$.
\end{proposition}

\section{Faster Rates with Variance Reduction}\label{sec:variace_reduction}

\begin{algorithm}[h]
	\caption{\algname{PAGER} (\algname{PAGE} with restarts)}\label{alg:PAGE_w_restarts}
	\begin{algorithmic}[1]
		\STATE Initialization: $\bar{x}_0,\bar{g}_0, K, \cb{ \Lambda_k = \rb{ \eta_k, T_k, p_k, b_k, b_k^{\prime}} : k=0,...,K-1 }$
		\FOR{$k=0, \ldots, K -  1$}
		\STATE $\rb{x_0, g_0} \leftarrow \rb{\bar{x}_k, \bar{g}_k}$
		\STATE $\rb{\eta, p, b, b^{\prime}} \leftarrow \rb{\eta_k, p_k, b_k, b_k^{\prime}}$
		\FOR{$t=0, \ldots, T_k -  1$}
    		\STATE $\xtpo = \xt - \eta g_t$
    		\STATE Sample $\chi \sim \text{Bernoulli}(p)$ 
    		\IF{$\chi=1$}
    		    \STATE $g_{t+1} = \frac{1}{b}\sum_{i=1}^{b} \nabla f_{\xi_{t+1}^i}(x_{t+1})$ 
    		\ELSE
    			\STATE $g_{t+1} = g_{t} +  \frac{1}{b^{\prime}}\sum_{i=1}^{b^{\prime}} \nabla f_{\xi_{t+1}^i}(x_{t+1})  - \frac{1}{b^{\prime}}\sum_{i=1}^{b^{\prime}} \nabla f_{\xi_{t+1}^i}(x_{t}) $ 
        	\ENDIF
		\ENDFOR
		\STATE $\rb{\bar{x}_{k+1}, \bar{g}_{k+1}} \leftarrow \rb{x_{t+1}, g_{t+1}}$
		\ENDFOR
		\STATE \textbf{Return: $\bar{x}_K$} 
	\end{algorithmic}
\end{algorithm}
To simplify the exposition of the results in this section, let us assume that $g_k(x_{t}, \xi_{t})$ is constructed explicitly via mini-batching 
$
g_k(x_t, \xi_t):= \frac{1}{b_k}\sum_{i=1}^{b_k} \nabla f_{\xi_t^i}(x_t) ,
$
where $\xi_t \eqdef \rb{\xi_t^1,\ldots, \xi_t^{b_k}}$ is a random vector of independent entries, $\xi_t$ are independent for all iterations, $\tiny{\{\nabla f_{\xi_t^i}(x_t) \}}_{i= 1}^{b_k}$ are queries provided by an oracle such that $\mathbb{E}[\nabla f_{\xi_t^i}(x_t) ]=\nabla f(x_t)$ and $\mathbb{E}[||\nabla f_{\xi_t^i}(x_t) -\nabla f(x_t)||^2]\leq \sigma^2$ for all $t\geq 0$. The variance of this estimator diminishes linearly in the size of the mini-batch $b_k$, i.e., $g_k(x_{t}, \xi_{t})$ satisfies
\begin{assumption}[$k$-BV, bounded variance]\label{as:UBV}
	Let Assumption~\ref{ass:k_ES} hold with $A = 0$, $B = 1$ and $C = \sigma^2$, i.e., $\Exp{\sqnorm{g_k(x, \xi) - \nabla f(x)} } \leq \fr{\sigma^2}{b_k} $.
\end{assumption}
Additionally, we assume that we have access to a gradient estimator {$g_k^{\prime}(x,\xi)$}, which satisfies the following
\begin{assumption}[Average $\cL$-smoothness (of order $k$)]\label{as:avg_smoothness_page}
    Let $g_k^{\prime}(x,\xi)\eqdef 	\frac{1}{b_k^{\prime}}\sum_{i=1}^{b_k^{\prime}} \nabla f_{\xi^i}(x)$ and $g_k^{\prime}(y,\xi) \eqdef \frac{1}{b_k^{\prime}}\sum_{i=1}^{b_k^{\prime}} \nabla f_{\xi^i}(y)$ be unbiased mini-batch estimators of the gradient of $f(\cdot)$ at points $x$ and $y$, respectively for shared stochasticity $\xi^i\sim \cD$ for each $i = 1,\dots,b_k^{\prime}$ and $\xi = (\xi^1,\ldots,\xi^{b_k})$. Define $\widetilde{\Delta}(x, y) :=  g_k^{\prime}(x,\xi) -  g_k^{\prime}(y,\xi)$. The average $\cL$-smoothness (of order $k$) holds if there exists $\cL \ge 0$ such that $\Exp{\sqnorm{\widetilde{\Delta}(x,y) - \Delta(x,y) } } \le \frac{\cL^2}{b_k^{\prime}} \sqnorm{ x - y } \quad  \text{for all } x, y \in \R^d,$
	where $\Delta(x, y) := \nabla f(x) - \nabla f(y)$.
\end{assumption}

\begin{remark}
The Assumption~\ref{as:avg_smoothness_page} holds in several standard settings. For instance, if each $\nabla f_{\xi^i}(x)$ is Lipschitz with constant $\bar{L}$ (almost surely or on average), then Assumption~\ref{as:avg_smoothness_page} holds with $\cL \leq \bar{L}$. Another example is when $f(\cdot)$ is of the form \eqref{eq:problem_finite_sum} and $b_k^{\prime} = n$, then $\cL = 0$. 
\end{remark}
\subsection{{PAGER} -- a new variance reduction for $\alpha$-P{\L} objectives}
We remark from the analysis of Algorithm~\ref{alg:sgd} in Section~\ref{sec:main} that merely playing with choice of $\eta_k$ and $b_k$ (chosen as polynomial functions of $k$) is not sufficient to improve the convergence, hence, we need to construct more sophisticated gradient estimator and reduce the variance using control variate. Now, we highlight the main algorithmic ingredients of our construction. First, let us describe the variance reduced estimator named \algname{PAGE}, which will be the main building block for our Algorithm~\ref{alg:PAGE_w_restarts}. \algname{PAGE} was introduced and analyzed in \citep{PAGE} and is known to be optimal for finding a first order stationary point. Moreover, it is easy to implement and designed via a small modification to mini-batch \algname{SGD}
$$g_{t+1} = \begin{cases}
				\frac{1}{b}\sum_{i=1}^{b} \nabla f_{\xi_{t+1}^i}(x_{t+1}) ,  & \text{w.p.} \quad\  p,\\
				g_{t} +  \frac{1}{b^{\prime}}\sum_{i=1}^{b^{\prime}} \rb{ \nabla f_{\xi_{t+1}^i}(x_{t+1}) - \nabla f_{\xi_{t+1}^i}(x_{t}) } , &\text{w.p.}\  1-p, 
			\end{cases}
$$
where $p$ is a small probability and mini-batch sizes satisfy $b > b^{\prime}$. 

However, while the method looks simple, the extension of its analysis to $\alpha$-P{\L} functions faces several difficulties.  \footnote{We refer the reader to Appendix~\ref{sec:variace_reduction_appendix}, where we explain the challenges in the analysis of variance reduction under $\al$-P{\L} condition and show how we overcome these difficulties using the restart strategy.}  Therefore, we introduce a new method, which we call \algname{PAGER} (Algorithm~\ref{alg:PAGE_w_restarts}) -- a \textit{Probabilistic Average Gradient Estimator with parameter Restart}. It takes as input the sequence of parameters $\cb{ \Lambda_k \eqdef \rb{ \eta_k, T_k, p_k, b_k, b_k^{\prime}} : k=0,...,K-1 }$, where $T_k$ is the length of stage $k$, $\eta_k, p_k, b_k, b_k^{\prime}$ step-size, probability, and batch-sizes at stage $k$. \algname{PAGER} updates this sequence of parameters in the outer loop $k = 0,\ldots, K-1$ and applies \algname{PAGE} estimator with a fixed set of parameters in the inner loop $t = 0,\ldots, T_k-1$. We will select $\cb{\Lambda_k}_{k\geq 0 }$ depending on the P{\L} power $\al$ to capture the dependence on the geometry of the problem and establish fast rates for each $\al$ in settings \eqref{eq:problem_online} and \eqref{eq:problem_finite_sum}. 

\subsection{Online case}
We present convergence guarantees for Algorithm~\ref{alg:PAGE_w_restarts} in the setting \eqref{eq:problem_online} and defer its formal proof to Appendix~\ref{sec:variace_reduction_appendix}.
\begin{theorem}\label{thm:PAGE_online_w_restarts1} 
	Let $f(\cdot)$ have the form \eqref{eq:problem_online} and satisfy Assumptions~\ref{ass:l-smooth}, \ref{as:lojasiewicz} (with $\al \in [1, 2)$), \ref{as:UBV} and \ref{as:avg_smoothness_page}, let the sequences\footnote{For brevity, in Theorem~\ref{thm:PAGE_online_w_restarts1} we define the input sequences up to constants hidden in $\Theta\rb{\cdot}$ notation. In fact, our analysis allows to specify these constants and we present detailed derivations in Appendix~\ref{sec:variace_reduction_appendix}. }  in Algorithm~\ref{alg:PAGE_w_restarts} be chosen as $b_k^{\prime} = \Theta\big( 2^{\fr{(2-\al)k}{\al}} \big)$, $p_k = \Theta\big( 2^{\fr{-(2-\al)k}{\al}} \big)$, $b_k = \Theta\big( 2^{\fr{2 k}{\al}} \big)$, $T_k = \Theta\big( 2^{\fr{(2-\al)k}{\al}} \big)$, $	\eta_k = \Theta\big( 1 \big) $.
    Then, for any $\epsilon_f > 0$ Algorithm \ref{alg:PAGE_w_restarts} returns a point $x$ with $\Exp{f(x) - f^{\star}} \leq \epsilon_f$ after 
$N := \sum_{k=0}^{K-1} T_k =  \cO  \big(\kappa \epsilon_f^{-\fr{2-\al}{\al} }\big) $ iterations, where $\kappa = \nfr{\cL}{\mu}$. The expected total computational cost (sample complexity) is $\cO\rb{ \rb{ \fr{\sigma^2 }{\mu} + \kappa^2 }  \epsilon_f^{-\fr{2}{\al} }} $.

\end{theorem}

\paragraph{Improvement over \algname{SGD}.} Theorem~\ref{thm:PAGE_online_w_restarts1} implies that \algname{PAGER} improves the sample complexity of \algname{SGD} from $\cO\big( \epsilon_f^{-\fr{4-\al}{\al} }\big)  $ to $\cO\big( \epsilon_f^{-\fr{2}{\al} }\big)  $ under $\al$-P{\L} condition for the whole spectrum of parameters $\al \in [1, 2)$. In the case $\al = 1$, which holds in many interesting applications (see Appendix~\ref{sec:examples_appendix} for examples), this leads to $\cO\big( \epsilon_f^{-2 }\big)$ sample complexity compared to the best known $\cO\big( \epsilon_f^{-3 }\big)$ for \algname{SGD}.
 \paragraph{Relation to convex optimization and last iterate convergence.} As a consequence of our analysis we obtain \textit{the optimal sample complexity for convex stochastic optimization} under the additional assumption that the iterates of the method remain bounded, i.e., $ \sqnorm{\xt - \xstar} \leq D $ for all $t\geq 0$, where $\xstar \in \argmin_x f(x) $.\footnote{Note that this assumption is mild since it holds for the iterates of \algname{PAGER}, for example, if we additionally assume that $f(\cdot)$ is coercive, i.e., $f(x) \rightarrow \infty$ for $x\rightarrow \infty$.} For $1$-P{\L} objectives, \algname{PAGER} has $\cO\rb{ \epsilon_f^{-2} }$ sample complexity.
 Since the iterates of the algorithm are bounded, convexity $\langle \nabla f(x), x - \xstar\rangle \geq f(x) - f(\xstar)$ implies $1$-P{\L} with $\mu = \fr{1}{2 D} $. This observation implies convergence of \algname{PAGER} for convex objectives with $\cO\rb{ \epsilon_f^{-2} }$ sample complexity, which is known to be non-improvable for convex stochastic optimization \citep{nemirovskij1983problem}.
 
 Moreover, we highlight that this result holds for the \textit{last iterate} of \algname{PAGER}, while the standard analysis of first order methods for convex functions guarantees convergence for the average iterate \citep{FirstorderSO_Lan}. The last iterate convergence for convex objectives was only recently established for \algname{SGD} by following an involved analysis with a careful control of iterates via suffix-averaging scheme  \citep{fontaine2021convergence}. 

\subsection{Finite sum case}\label{subsec:finit_sum}
Let $f(\cdot)$ have the finite sum form \eqref{eq:problem_finite_sum}. Then we obtain the following result.
\begin{theorem}\label{thm:PAGE_w_restarts_finite_sum}
    Let $f(\cdot)$ have the form \eqref{eq:problem_finite_sum} and satisfy Assumptions \ref{ass:l-smooth}, \ref{as:lojasiewicz} (with $\al \in [1, 2)$) and \ref{as:avg_smoothness_page}, let the sequences be chosen as $p_k = \fr{1}{n+1}$, $b_k^{\prime} = 1$, $b_k = n$, $T_k = \Theta\big( 2^{\fr{(2-\al)k}{\al}} \big)$, $\eta_k = \Theta\big( 1 \big) $.
    Then, for any $\epsilon_f > 0$, Algorithm \ref{alg:PAGE_w_restarts} returns a point $x$ with $\Exp{f(x) - f^{\star}} \leq \epsilon_f$ after $N := \sum_{k=0}^{K-1} T_k =  \cwO\big( n + \sqrt{n} \kappa  \epsilon_f^{-\fr{2-\al}{\al} }\big)$ iterations, where $\kappa = \nfr{\cL}{\mu}$ The expected total computational cost (sample complexity) is $\cwO\big( n + \sqrt{n} \kappa \epsilon_f^{-\fr{2-\al}{\al}} \big) $. 
\end{theorem}

The proof is deferred to Appendix~\ref{sec:variace_reduction_appendix}. Theorem~\ref{thm:PAGE_w_restarts_finite_sum} quantifies the improvement of \algname{PAGER} over \algname{GD} in the finite sum setting in terms of $n$ and over \algname{SGD} in terms of $\epsilon_f$, see Table~\ref{tbl:complexity_f} for comparison. Recall that \algname{GD} has sample complexity $\cO\big( n \kappa \epsilon_f^{\fr{-(2-\al)}{\al}} \big)$. When $n$ is large, we get the improvement of order $\sqrt{n}$. Notice that in the limit $\al \rightarrow 2$, it matches the best known result for $2$-P{\L} objectives \cite{PAGE}.

\section{Conclusion}

We analyzed the complexity of \algname{SGD} when the objective satisfies global K{\L} inequality and the queries  from stochastic gradient oracle satisfy weak expected smoothness. 
We introduced a general framework for this analysis which resulted in a sample complexity of $\mathcal{O}(\epsilon^{-(4-\al)/\al})$ for \algname{SGD} with objectives satisfying $\al$-P{\L} condition. 
We also demonstrated the tightness of this rate under the specific choice of stepsizes. Last but not least, we developed a modified \algname{SGD} with variance reduction and restarting (\algname{PAGER}), which improves the sample complexity of \algname{SGD} for the whole spectrum of parameters $\al \in [1,2)$ and achieves the optimal rate for the important case of $1$-P{\L} objectives. 

\section*{Acknowledgements}
We would like to thank Anas Barakat and Anastasia Kireeva for valuable discussions. The work of I.~Fatkhullin was supported by ETH AI Center doctoral fellowship.


\bibliographystyle{plain}
\bibliography{ref_sgd.bib}

\clearpage
\appendix
\newpage

\tableofcontents
\newpage

\begin{center}
    \bfseries\Large Appendix
\end{center}

\section{Examples}\label{sec:examples_appendix}
\subsection{$\al$-P{\L} Functions}

In this section, we provide some examples and applications of global K{\L} functions. Particularly, we focus on the class of  $\al$-P{\L} functions with $\al \in [1, 2]$.  We start with simple one dimensional functions. 
\begin{example}
	Consider $f(x) = c \cdot |x|^q$, where $q > 1$, $c > 0$. $f(x)$ satisfies Assumption~\ref{as:lojasiewicz} with $\al = \fr{q}{q - 1}$ and $\mu = \fr{c^{\nfr{2}{q}} q^2}{2}$. 
\end{example}

\begin{example}
	Consider $f(x) = \fr{e^x + e^{-x}}{2} - 1$. $f(x)$ satisfies Assumption~\ref{as:lojasiewicz} with $\al=1$ and $\mu = \nfr{1}{2}$. 
\end{example}

\begin{example}
	Consider $f(x) = cosh(x) + 8 \cdot cosh(sin(x)) - 9$, where $cosh(x) = \nfr{(e^x + e^{-x})}{2}$. The derivative is $f^{\prime}(x) = sinh(x) + 8 \cdot cos(x) \cdot sinh(sin(x))$ and $| f^{\prime}(x) | \geq 10^{-2} \cdot f(x)$ for all $x$. Then $f(x)$ satisfies Assumption~\ref{as:lojasiewicz} with $\al=1$ and $\mu = 5 \cdot 10^{-5}$. 
\end{example}

Note that the functions in Example 1 and Example 2 are convex, whereas the function in Example 3 is nonconvex. 



The following proposition shows that  K{\L} property is preserved under some operators such as direct addition. 
\begin{proposition}\label{prop:separable_KL}
    Let $f(\cdot)$ be a separable function, i.e., $f(x) \eqdef \fr{1}{n} \sum_{i=1}^n f_i(x_i)$, where $x = (x_1,\dots,x_n)$, $x_i \in\R^{d_i}$, $\sum_{i=1}^n d_i = d$. Let each $f_i(\cdot)$ satisfy K{\L} inequality (Assumption~\ref{ass:kl_our}) with $\phi_i(t)$. Then $f(\cdot)$ also satisfies K{\L} inequality with $\phi(t) \eqdef \fr{1}{\sqrt{n}} \min_{1\leq i\leq n} \phi_i(t)$.
\end{proposition}
\begin{proof}
By separability and K{\L} condition we have
\begin{eqnarray}
    \sqnorm{\nabla f(x)} &=& \sum_{i=1}^n \fr{1}{n^2} \sqnorm{\nabla f_i(x_i)} \notag \\
    &\geq& \sum_{i=1}^n \fr{1}{n^2} \phi_i^2\rb{ f_i(x_i) - f_i^{inf}}
    \notag \\
    &\overset{(i)}{\geq}& \fr{1}{n} \sum_{i=1}^n  \phi^2\rb{ f_i(x_i) - f_i^{inf}}
    \notag \\
    &\overset{(ii)}{\geq}&  \phi^2\rb{\fr{1}{n} \sum_{i=1}^n  f_i(x_i) - f_i^{inf}}
     \notag \\
    &\overset{(iii)}{\geq}&  \phi^2\rb{ f(x) - f^{inf}},
\end{eqnarray}
where $(i)$ holds by definition of $\phi(t)$, $(ii)$ is due to convexity of $\phi(t) \eqdef \fr{1}{\sqrt{n}} \min_{1\leq i\leq n} \phi_i(t)$ and Jensen's inequality and $(iii)$ follows from $\fr{1}{n} \sum_{i=1}^n \inf_{x_i} f_i(x_i) \leq \inf_{x} \fr{1}{n} \sum_{i=1}^n f_i(x)$ for any $x = (x_1,\dots,x_n)$.
\end{proof}

The above Proposition~\ref{prop:separable_KL} implies, in particular, that if we have a separable function $f(x) = \sum_{i=1}^{n} f_i(x_i)$ and each $f_i(x_i)$ is $1$-P{\L} with $\mu_i$, $i = 1,\ldots,n$, then $f(x)$ satisfies $1$-P{\L} with $\mu = \fr{\mu_{min}}{n}$ .
\begin{example}
	Consider $f(x,y) = cosh(x) + 8 \cdot cosh(sin(x)) + 0.5 \cdot cosh(y) + 2.5 \cdot cosh(sin(y)) - 12$. This function of two variables satisfies Assumption~\ref{as:lojasiewicz} with $\al=1$ and $\mu = 5 \cdot 10^{-5}$. 
\end{example}

 Now we list several problems which occur in applications and satisfy $\al$-P{\L} with $\al = 1$.
 \begin{example}[Policy gradient optimization in RL]\label{ex:PO_RL}
     Consider a Markov Decision Process (MDP) $M=\{\mathcal{S}, \mathcal{A}, \mathcal{P}, \mathcal{R}, \gamma, \rho\}$, where $\mathcal{S}$ is a state space; $\mathcal{A}$ is an action space; $\mathcal{P}$ is a transition model, where $\mathcal{P}(s^{\prime} | s, a)$ is the transition density to state $s^{\prime}$ from a given state $s$ under a given action $a$; $\mathcal{R} = \mathcal{R}(s,a)$ is the bounded reward function for state-action pair $(s,a)$; $\gamma \in [0,1)$ is the discount factor; and $\rho$ is the initial state distribution. The behavior of the agent in MDP is characterized by the parametric policy $\pi_\theta(a|s)$ over $\mathcal{S}\times\mathcal{A}$, which denotes the probability of taking action $a$ at the state $s$. The policy $\pi_\theta$ is assumed to be differentiable with respect to parameter $\theta\in\R^d$. Let $\tau = \cb{s_t, a_t}_{t\geq0}$ be a trajectory generated by the policy $\pi_\theta$ and it is distributed according to distribution $\tau \sim p(\tau | \pi_\theta)$. The expected return of the policy $\pi_\theta$ is defined by  
     $$
     J\left({\theta}\right) \eqdef \mathbb{E}_{\tau}\left[\sum_{t=0}^{\infty} \gamma^{t} \mathcal{R}\left(s_{t}, a_{t}\right)\right].
     $$
     The goal of policy-based methods is to find $\theta$ which maximizes the expected return $\theta^{\star} \in \argmax_{\theta} J(\theta)$. It was recently shown that the above objective satisfies $1$-P{\L} assumption
     $$
     \norm{\nabla J(\theta)} \geq \sqrt{2\mu} \rb{J^{\star} - J(\theta)} \qquad \text{for all } \theta \in \R^d
     $$
     under the standard assumptions on $\pi_\theta$ and $\rho$ such as non-degenerate Fisher matrix and transferred compatible function approximation error \citep{Mei_SoftMax_Entropy_PG,Agarwal_TheoryPolicyGradient2020,yuan2021general}.
 \end{example}
 
 \begin{example}[Operations management problems]\label{ex:OR_problem} 
 In applications such as supply chain or revenue management \citep{Chen_OR_2022}, problems can often be formulated as 
     \begin{equation}\label{eq:OR_problem}
     \min_{x\in \mathcal{X}} F(x) \eqdef \Exp{ \phi(x \wedge \xi) },
     \end{equation}
     where $\mathcal{X} $ is a convex compact subset of $\R^d$, $\xi$ is a random vector,  $\wedge$ denotes a component-wise minimum and $\phi(\cdot)$ is convex. As a result, $F(\cdot)$ becomes non-convex. On the other hand, such problem often admits a convex reformulation 
      \begin{equation}\label{eq:cvx_reform}
     \min_{y\in \mathcal{Y}} G(y) \eqdef F(g^{-1}(y)) ,
     \end{equation}
      where $g(x)=\Exp{x\wedge\xi}$ and function $G(\cdot)$ is convex. Suppose $g: \mathcal{X} \rightarrow \mathcal{Y}$ is a bijective differentiable map with $\nabla g(x) \succeq \lambda I$, $\lambda > 0$ for all $x \in \mathcal{X}$, then 
     function $F(\cdot)$ satisfies $1$-P{\L} condition. This is because: for any $x$ with $g(x)=y$,
\begin{eqnarray}
     F(x) - F(\xstar) &=& G(y) - G(y^{\star}) \notag\\
     &\leq& \langle \nabla G(y) , y - y^{\star} \rangle \notag\\
     &\leq&  \norm{\nabla G(y)} \norm{ y - y^{\star}}  \notag\\
     &=&  \norm{\nabla g^{-1}(y) \nabla F(x)} \norm{ y - y^{\star}}  \notag\\
      &\leq & \fr{D_{\mathcal{Y}}}{\lambda} \norm{ \nabla F(x)} , \notag
\end{eqnarray}
where $D_{\mathcal{Y}}$ is the diameter of the set $\mathcal{Y}$. Therefore, $F(\cdot)$ is $1$-P{\L} with $\mu = \fr{1}{2} \fr{\lambda^2}{D_{\mathcal{Y}}^2}$.

 \end{example}

\begin{remark}
Note that even though the problem in Example~\ref{ex:OR_problem} satisfies $1$-P{\L} condition, our theory developed in this work is not directly applicable to solve this problem The reason is that this problem has a compact constraint and therefore requires an appropriate generalization of P{\L} condition, e.g., using the notion of gradient mapping or the subgradient of the indicator function of the set $\mathcal{X}$, see \cite{karimi2016linear} for examples. However, our theory becomes applicable for this problem if we additionally assume that the solution of \eqref{eq:OR_problem} lies in the interior of $\mathcal{X}$ and all the iterates $\cb{x_t}_{t\geq 0}$ generated by the method remain in the interior of $\mathcal{X}$. 
\end{remark}

\subsection{K{\L} Functions}
\begin{example}
A commonly used type of loss function in machine learning applications is a squared cross entropy (CE), it is given by 
\begin{align*}
    \ell(x,y):=\sum_i y_i\log\Big(\frac{e^{x_i}}{\sum_j e^{x_j}}\Big)^2.
\end{align*}
Under such loss function, it is known \cite{scaman2022convergence} that K{\L} condition holds with corresponding function $\phi(t)=\min\{t,\sqrt{t}\}$. This is function is both positive and $\phi(t)^2$ is convex. Next, we apply the result of Theorem \ref{lemma_man} to obtain the convergence rate of \algname{SGD} for this type of loss functions assuming the stochastic gradient estimator satisfying Assumption \ref{ass:k_ES} with $h(t)=t$. 
First step is to obtain the stationary point $r(\eta)$ using Equation \eqref{eq:statioanry}.
\begin{align*}
    2a\eta^2t+2\frac{d\eta^2}{b}=\eta\big(\min\{t,\sqrt{t}\}\big)^2.
\end{align*}
It is straightforward to see that for small enough $\eta$, the stationary point is smaller than 1. In this case, $\min\{t,\sqrt{t}\}$ is $t$. Therefore, we are in the setting of Corollary \ref{corr:main1} with $\alpha=1, \beta=1$, and $\tau=0$. This implies that the interation (and sample) complexity of \algname{SGD} is of the order $\mathcal{O}(\epsilon_f^{-3})$. Moreover, if $A$ in Assumption~\ref{ass:k_ES} is zero, using a similar argument and the result of Theorem~\ref{thm:PAGE_online_w_restarts1}, one can derive that \algname{PAGER} give us $\mathcal{O}(\epsilon_f^{-2})$ sample complexity.
\end{example}
To illustrate the generality of the result of Theorem \ref{lemma_man}, next we present the convergence rate of \algname{SGD} for objective functions that satisfy the global K{\L} condition with $\phi(t)=\sqrt{t\log(t+1)}$ under Assumption~\ref{ass:k_ES} with $h(t)=\log(1+t)$.
\begin{example}
Consider the scenario in which the objective function satisfies the global K{\L} condition with $\phi(t)=\sqrt{t\log(t+1)}$ and a stochastic gradient estimator satisfies Assumption~\ref{ass:k_ES} with $h(t)=\log(1+t)$. In this case,  Equation \eqref{eq:statioanry} becomes
\begin{align*}
    2a\eta^2\log(1+t)+2\frac{d\eta^2}{b}=\eta t\log(1+t).
\end{align*}
Defining $u:=\log(t+1)$ yields
\begin{align*}
    \eta\big(2a u+2\frac{d}{b}\big)= ( e^{u}-1)u\approx (u+\frac{u^2}{2})u.
\end{align*}
The last approximation is true since for small enough $\eta$, $u$ is less than one. Solving the above cubic equation leads to a solution that is of the order $u=\Theta(\sqrt{\eta})$ or equivalently $r(\eta)=\Theta(\exp(\sqrt{\eta})-1)$. Note that for small enough $\eta\ll 1$, we have $\Theta(\exp(\sqrt{\eta})-1)=\Theta(\sqrt{\eta}+\eta)=\Theta(\sqrt{\eta})$, i.e., $\nu=0.5$. To obtain $\zeta$ in Theorem \ref{lemma_man}, we use Equation \eqref{eq:condition1} which leads to
\begin{align*}
    &1+\frac{a\eta^2_k}{1+r(\eta_k)} -\frac{\eta_k}{2}\Big(\frac{r(\eta_k)}{1+r(\eta_k)}+\log(1+r(\eta_k))\Big)\\
    &=1+\frac{a\eta^2_k}{1+\sqrt{\eta_k}} -\frac{\eta_k}{2}\Big(\frac{\sqrt{\eta_k}}{1+\sqrt{\eta_k}}+\log(1+\sqrt{\eta_k})\Big)=1-\omega_k k^{-1},
\end{align*}
In order to have the above equality, we can have $\eta_k=\Theta(k^{-1})$. Finally, the result of Theorem \ref{lemma_man} yields $\delta_k=\mathcal{O}(k^{-\zeta\nu})=\mathcal{O}(k^{-0.5})$.
\end{example}

\section{Proofs for Section \ref{sec:main}}

\subsection{Proof of Lemma \ref{lemma:dynamic_inner}}
\textbf{Lemma 1.} 
Under Assumptions \ref{ass:l-smooth}, \ref{ass:kl_our}, and \ref{ass:k_ES} with constant cost, i.e., $b := b_k$,  we obtain
\begin{align*}
         \delta_{t+1}\leq \delta_t+a\eta^2\cdot h\big(\delta_t\big)-\frac{\eta}{2}\phi^2(\delta_t)+ \fr{d\eta^2}{b},
\end{align*}
where $\delta_t := \Exp{f(x_t) - f^{\star}}$, $a := L A$, $d := \fr{L C}{2 }$, $\eta := \eta_k$.

\begin{proof}
Let $\{x_0,x_1,x_2,...\}$ denote the sequence of points that are obtained from SGD. From the L-smoothness assumption, we obtain
\begin{align*}
    f(x_{t+1})\leq f(x_t) -\eta \langle\nabla f(x_t),g_k(x_t,\xi_t)\rangle + \frac{L}{2}||x_{t+1}-x_t||^2.
\end{align*}
Taking the conditional expectation of both side of the above inequality given $x_t$ yields
\begin{align*}
    & \mathbb{E}[f(x_{t+1})-f(x_t)|x_t]\leq -\eta \mathbb{E}\big[\langle\nabla f(x_t),g_k(x_t,\xi_t)\rangle|x_t\big] + \frac{L}{2}\eta^2\mathbb{E}\big[||g_k(x_t,\xi_t)||^2|x_t\big].
\end{align*}
Using Assumption \ref{ass:k_ES} and the fact that oracle's queries are unbiased, we obtain
\begin{align*}
    \mathbb{E}[f(x_{t+1})\!-\!f(x_t)|x_t]\leq  LA\eta^2\cdot h\big(f(x_t)\!-\!f^*\big)-\eta(1-\!\frac{L}{2}\eta B)\cdot\phi^2\big(f(x_t)\!-\!f^*\big)+\frac{L}{2}\eta^2 \frac{C}{b},
\end{align*}
where $b=b_k$ denotes the cost of gradient $g_k$. 
Since the choice of learning rate is ours, we select it such that $(1-\frac{L}{2}\eta B)\geq\frac{1}{2}$. Using Assumption \ref{ass:kl_our} for points around the optimum point $x^*$, we obtain
\begin{align*}
    & \mathbb{E}[\varrho_{t+1}|x_t]-\varrho_t\leq a\eta^2\cdot h\big(\varrho_t\big)-\frac{\eta}{2}\phi^2(\varrho_t)+\frac{d\eta^2}{b},
\end{align*}
where $\varrho_t:=f(x_t)-f^*$, $a:= LA$, and $d:=\frac{LC}{2}$.
Let $\delta_t=\mathbb{E}[\varrho_t]$.
Using the fact $h(t)$ is concave and $\phi^2$ is convex, and Jensen's inequality, we obtain the result. 
\end{proof}


\subsection{Proof of Theorem \ref{lemma_man}}
We first prove the following technical lemma. 

\begin{lemma}\label{ll:12}
Consider a series $\{r_t\}_{t\geq0}$ that for every integer $T>0$ satisfies the following inequality
\begin{align*}
    r_{t}\leq \prod_{i=1}^k(1-a_i i^{-1})^{T} r_0 + \mathcal{O}(k^{-b}),
\end{align*}
where $t=kT$ and $a<a_i<A\leq 1$ for some positive constants $a$ and $A$ and all $i$. Then, there exists $T$ such that $r_{t}=\mathcal{O}(t^{-b})$.
\end{lemma}
\begin{proof}
Using the fact that $a_i$s are bounded and $1-x\leq \exp(-x)$, we obtain
\begin{align*}
    \prod_{i=1}^k(1-a_i i^{-1})^{T} \leq \exp\big(-aT\sum_{i=1}^k i^{-1}\big)\leq (k+1)^{-aT}.
\end{align*}
In the above inequality, we used $\sum_{i=1}^k i^{-1}\geq \int_1^{k+1}x^{-1}dx=\log(k+1)$. Selecting $T=\lceil b/a \rceil$ will imply the result. 
\end{proof}

\textbf{Theorem 1.}
Suppose there exist $\nu, \{\omega_j\}_{j\geq0}$, and $\zeta\geq0$ such that $\eta_k=\Theta(k^{-\zeta})$, $r(\eta_k)=\Theta(k^{-\zeta\nu})$, $|1-\omega_k|<1$, and
\begin{align*}
    &1+a\eta^2_k h'\big(r(\eta_k)\big)-\eta_k\phi'\big(r(\eta_k)\big)\phi\big(r(\eta_k)\big)=1-\omega_k k^{-1}.
\end{align*}
Then, $\delta_k=\mathcal{O}(k^{-\zeta\nu})$ and the iteration complexity of  Algorithm~\ref{alg:sgd} is $\mathcal{O}(\epsilon_f^{-1/(\zeta\nu)})$. 

\begin{proof}
Suppose, we are in the $k$ iteration of the outer-loop of Algorithm \ref{alg:sgd}. 
Using Lemma \ref{lemma:dynamic_inner} and the definition of $r(\eta)$ in \eqref{eq:statioanry}, we have
\begin{align*}
         \delta_{t+1}\leq \delta_t + a\eta_k^2 \Big(h\big(\delta_t\big)-h\big(r(\eta_k)\big)\Big)-\frac{\eta_k}{2}\Big(\phi^2(\delta_t)-\phi^2\big(r(\eta_k)\big)\Big).
\end{align*}
By defining $y_t:=\delta_t-r(\eta_k)$ and using the concavity of functions $h(\cdot)$ and $-\phi^{2}(\cdot)$, we obtain 
\begin{align*}
         y_{t+1}&\leq 
         y_t\Big(1+a\eta_k^2  h'\big(r(\eta_k)\big)-\eta\phi'\big(r(\eta_k)\big)\phi\big(r(\eta_k)\big)\Big).
\end{align*}
Given the assumption in Theorem \ref{lemma_man}, we have
\begin{align*}
    y_{t+1}\leq y_t \big(1-\omega_k k^{-1}\big).
\end{align*}
Recall that  $k$ corresponds to the index of the outer-loop. After $t$ iterations of the inner-loop (in which index $k$ is fixed), we obtain
\begin{align}\label{eq:aftercondition}
    y_{t}\leq y_0 \big(1-\omega_k k^{-1}\big)^t.
\end{align}

This shows the rate at which the inner-loop of Algorithm \ref{alg:sgd} (lines 4-5) converges to point $x$, where $r(\eta_k)=f(x)-f^*$.
Based on Equation \eqref{eq:aftercondition}, after setting $\eta=\eta_1$ and $T_1$ rounds of  the inner-loop, we obtain
\begin{align*}
    y_{T_1}\leq y_0 \big(1-\omega_1\big)^{T_1} \Rightarrow \delta_{T_1}\leq y_0 \big(1-\omega_1\big)^{T_1}+r(\eta_1). 
\end{align*}
Continuing this process, after updating $\eta=\eta_2$ and going through the inner-loop for another $T_2$ iterations imply
\begin{align*}
    y_{T_1+T_2}\leq  (\delta_{T_1}-r(\eta_2))\big(1-\frac{\omega_2}{2}\big)^{T_2} \Rightarrow \delta_{T_1+T_2}\leq  \big(y_0 \big(1-\omega_1\big)^{T_1}+r(\eta_1)-r(\eta_2)\big)\big(1-\frac{\omega_2}{2}\big)^{T_2} +r(\eta_2). 
\end{align*}
The above inequality is because before starting the inner-loop for the second round, the initial point for $y$ is $y_{T_1}:=\delta_{T_1}-r(\eta_2)$.
Using induction and after $k$ rounds of outer-loop, we obtain
\begin{align*}
    \delta_{t}&\leq  \prod_{i=1}^k\big|1-\frac{\omega_i}{i}\big|^{T_i}y_0+\prod_{i=1}^k\big|1-\frac{\omega_i}{i}\big|^{T_i}\left(\sum_{j=1}^{k-1}\frac{r(\eta_j)-r(\eta_{j+1})}{\prod_{i'=1}^j\big(1-\frac{\omega_{i'}}{i'}\big)^{T_{i'}}}\right)+r(\eta_{k})\\
    &=\prod_{i=1}^k\big|1-\frac{\omega_i}{i}\big|^{T_i}y_0+\sum_{j=1}^{k-1}\big(r(\eta_j)-r(\eta_{j+1})\big)\prod_{i=j+1}^k\big|1-\frac{\omega_i}{i}\big|^{T_i}+r(\eta_{k})
\end{align*}
where $t=\sum_{j} T_j$. 
In Algorithm \ref{alg:sgd}, $T_i$ are selected to be $T$. 
Next, using Lemma \ref{ll:12}, we show there exist a positive constant $T$ such that $\delta_t=\mathcal{O}(t^{-\nu\zeta})$.
Following the proof of Lemma \ref{ll:12}, we have
\begin{align*}
    \delta_{t}&\leq\prod_{i=1}^k\big|1-\frac{\omega_i}{i}\big|^{T}y_0+\sum_{j=1}^{k-1}\big(r(\eta_j)-r(\eta_{j+1})\big)\prod_{i=j+1}^k\big|1-\frac{\omega_i}{i}\big|^{T}+r(\eta_{k})\\
    &\leq (k+1)^{-\omega T} y_0+\sum_{j=1}^{k-1}\big(r(\eta_j)-r(\eta_{j+1})\big)\left(\frac{j+1}{k+1}\right)^{\omega T}+r(\eta_{k}),
\end{align*}
where $t=kT$ and $\omega=\min_i \omega_i$.
Let $b:=\nu\zeta$ and $T:=\lceil (b+1)/\omega\rceil$.
Since $r(\eta)=\Theta(\eta^\nu)$ and $\eta_j=\Theta(j^{-\zeta})$, then there exists a constant $C>0$ such that
\begin{align*}
    \delta_{t}&\leq (k+1)^{-\omega T} y_0+C\sum_{j=1}^{k-1}\big(j^{-b}-(j+1)^{-b}\big)\left(\frac{j+1}{k+1}\right)^{\omega T}+\mathcal{O}(k^{-b})\\
    &\leq \mathcal{O}(k^{-b})+\frac{C}{(k+1)^{b+1}}\sum_{j=1}^{k-1}\Big((1+\frac{1}{j})^b-1\Big)(j+1)+\mathcal{O}(k^{-b}).
\end{align*}
Using $(1+x)^b-1\leq bx/(1-bx)$ for $x<1/b$ and $b\geq0$, we obtain
\begin{align*}
    \delta_{t}&\leq \mathcal{O}(k^{-b})+\frac{C'}{(k+1)^{b+1}} +\frac{Cb}{(k+1)^{b+1}}\sum_{j>b}^{k-1}\frac{j+1}{j-b}+\mathcal{O}(k^{-b})=\mathcal{O}(k^{-b}).
\end{align*}
where $C'\geq0$ is a constant corresponding to the part of the summation for $j\leq b$. The result follows from the fact that $k=t/T$ and $T$ is a constant.
\end{proof}


\subsection{Proof of Corollary \ref{corr:main1}}
\textbf{Corollary 1.} 
Consider a special case of Assumption~\ref{ass:k_ES} with $h(t)=t^\beta$ and $b_k=k^{\tau}$, where $\beta\in(0,1]$ and $\tau\geq0$. 
Suppose the objective function $f$ satisfies Assumptions \ref{ass:l-smooth} and \ref{as:lojasiewicz}.
Let $\gamma:=\alpha\beta$.
Then, for any $\epsilon_f > 0$, Algorithm \ref{alg:sgd} returns a point $x$ with $\Exp{f(x) - f^{\star}} \leq \epsilon_f$ after $N := K \cdot T $ iterations.
\\
i) When $\gamma=2$ ($\alpha=2$ and $\beta=1$), we have
\begin{align*}
   & N = \mathcal{O}(\epsilon_f^{-\frac{1}{1+\tau}}), \ \text{with}\ \eta_k=\Theta(k^{-1}).
\end{align*}

ii) When $\gamma<2$, we have
\begin{align*}
    & N = \mathcal{O}\big(\epsilon_f^{-\frac{4-\alpha}{\alpha(\tau+1)}}\big) \ \text{with}\ \eta_k=\Theta(k^{-\frac{\tau+1}{2-\alpha/2}+\tau})\ \text{if}\  \tau\leq \frac{\gamma}{4-\alpha-\gamma}, and \\
    & N = \mathcal{O}\big(\epsilon_f^{-\frac{4-\alpha-\gamma}{\alpha}}\big) \ \text{with}\ \eta_k=\Theta(k^{-\frac{2-\gamma}{4-\alpha-\gamma}})\ \text{if}\ \tau> \frac{\gamma}{4-\alpha-\gamma}.
\end{align*}

\begin{proof}
Using the result of Theorem \ref{lemma_man}, we need to specify the constants $\nu$ and $\zeta$. 
To do so, we first characterize the stationary point for the  special setting of this corollary. 
Equation \eqref{eq:statioanry} becomes 
\begin{align}\label{eq:eq:ap1}
    a\eta\cdot t^\beta+\frac{d\eta}{b}=\mu t^{2/\alpha}.
\end{align}

Let 
$\gamma:=\alpha\beta$ and define the following function
\begin{align*}
    H_{\eta}(t):=a\eta^2t^{\beta} - \mu\eta t^{\frac{2}{\alpha}}+\frac{d\eta^2}{b}.
\end{align*}
Next, we either find $r(\eta)$ exactly or bound it.
Depending on whether $\gamma$ is less than or equal to $2$, the analysis of $H_\eta(t)=0$ is different. We study each case separately. 
\\
\\
\textbf{I)} $\gamma=2$ (or $\beta=1$ and $\alpha=2$):\  
In this case, we can find $r(\eta)$ exactly and it is given by
\begin{align*}
    r(\eta)=\frac{\frac{d\eta}{b}}{\mu-a\eta}=\Theta\rb{\frac{\eta}{b}}=\Theta\big(k^{-\tau}\eta\big).
\end{align*}
Note that in the above expression, we used the fact that $b_k=\Theta(k^\tau)$.
Next is to find the parameters in Theorem \ref{lemma_man}. To do so, from Equation \ref{eq:condition1} with $h(t)=t$ and $\phi(t)=\sqrt{2\mu t}$, we have
\begin{align*}
     &1+a\eta^2_k h'\big(r(\eta_k)\big)-\eta_k\phi'\big(r(\eta_k)\big)\phi\big(r(\eta_k)\big)=1+a\big(\Theta(k^{-\zeta})\big)^2  -\mu\Theta(k^{-\zeta})=1-\omega_k k^{-1}.
\end{align*}
In order to have the above equality, we should have $\zeta=1$. 
Now, suppose that $\tau\geq0$, then $r(\eta)=\Theta(k^{-(1+\tau)})$ and based on Theorem \ref{lemma_man}, we obtain the convergence rate of $\mathcal{O}(k^{-(1+\tau)})$ for $\delta_k$.

\textbf{II) }$0\leq \gamma<2$:\ 
In this case, we find lower and upper bound for $r(\eta)$. 
To this end, consider the following point for some constant $S$, 
$$
t_0:=\rb{\frac{\frac{d\eta}{b}+S\frac{\eta}{b}}{\mu}}^{\frac{\alpha}{2}}=\Theta\rb{\rb{\frac{\eta}{b}}^{\frac{\alpha}{2}}}.
$$
For this point,  we have
\begin{align*}
    H_\eta(t_0)=a\frac{\eta^{2+\frac{\gamma}{2}}}{b^{\frac{\gamma}{2}}}\Big(\frac{d+S}{\mu}\Big)^{\frac{\gamma}{2}} - S\frac{\eta^2}{b}.
\end{align*}
For $S=0$,  $H_\eta(t_0)>0$. 
On the other hand, if $\eta=\Theta(k^{-\zeta})$ and $b_k=\Theta(k^{\tau})$, then for $(\tau+\zeta)\frac{\gamma}{2}\geq \tau$ and large enough $S$, we have $H_\eta(t_0)<0$.
This implies 
\begin{align*}
    r(\eta)=\Theta\rb{\rb{\frac{\eta}{b}}^{\frac{\alpha}{2}}}.
\end{align*}
Next is to check whether \eqref{eq:condition1} holds for $\eta_k=\Theta(k^{-\zeta})$, $b_k=\Theta(k^{\tau})$, $h(t)=t^\beta$, and $\phi(t)=\sqrt{2\mu}t^{2/\alpha}$, i.e., 
\begin{align*}
     &1+a\beta\big(\Theta(k^{-\zeta})\big)^2  \Big(\Theta( k^{-(\zeta+\tau)\alpha/2})\Big)^{\beta-1}-\frac{2\mu}{\alpha}\Theta(k^{-\zeta})\Big(\Theta( k^{-(\zeta+\tau)\alpha/2})\Big)^{2/\alpha-1}=1-\omega_k k^{-1}.
\end{align*}
The order of the first term is $\mathcal{O}(k^{-(2\zeta+(\zeta+\tau)\alpha(\beta-1)/2)})$ 
and the order of the second term is
$\mathcal{O}(k^{-(\zeta+(\zeta+\tau)\alpha(2/\alpha-1)/2)})$. 
In order for the above expression to hold, we should have 
\begin{align}\label{eq:app:main:1}
    \zeta+(\zeta+\tau)\alpha(2/\alpha-1)/2\leq 2\zeta+(\zeta+\tau)\alpha(\beta-1)/2,
\end{align}
and 
\begin{align}\label{eq:app:main:2}
    \zeta+(\zeta+\tau)\alpha(2/\alpha-1)/2\leq 1.
\end{align}
Inequality \eqref{eq:app:main:1} implies $\gamma\zeta\geq(2-\gamma)\tau$ and inequality \eqref{eq:app:main:2} leads to $\zeta<\frac{\tau+1}{2-\alpha/2}-\tau$.
See Figure \ref{fig:region1} for an example of the region $(\zeta,\tau)$ for which both \eqref{eq:app:main:1} and \eqref{eq:app:main:2} hold.
\begin{figure} 
    \centering
    \includegraphics[scale=.4]{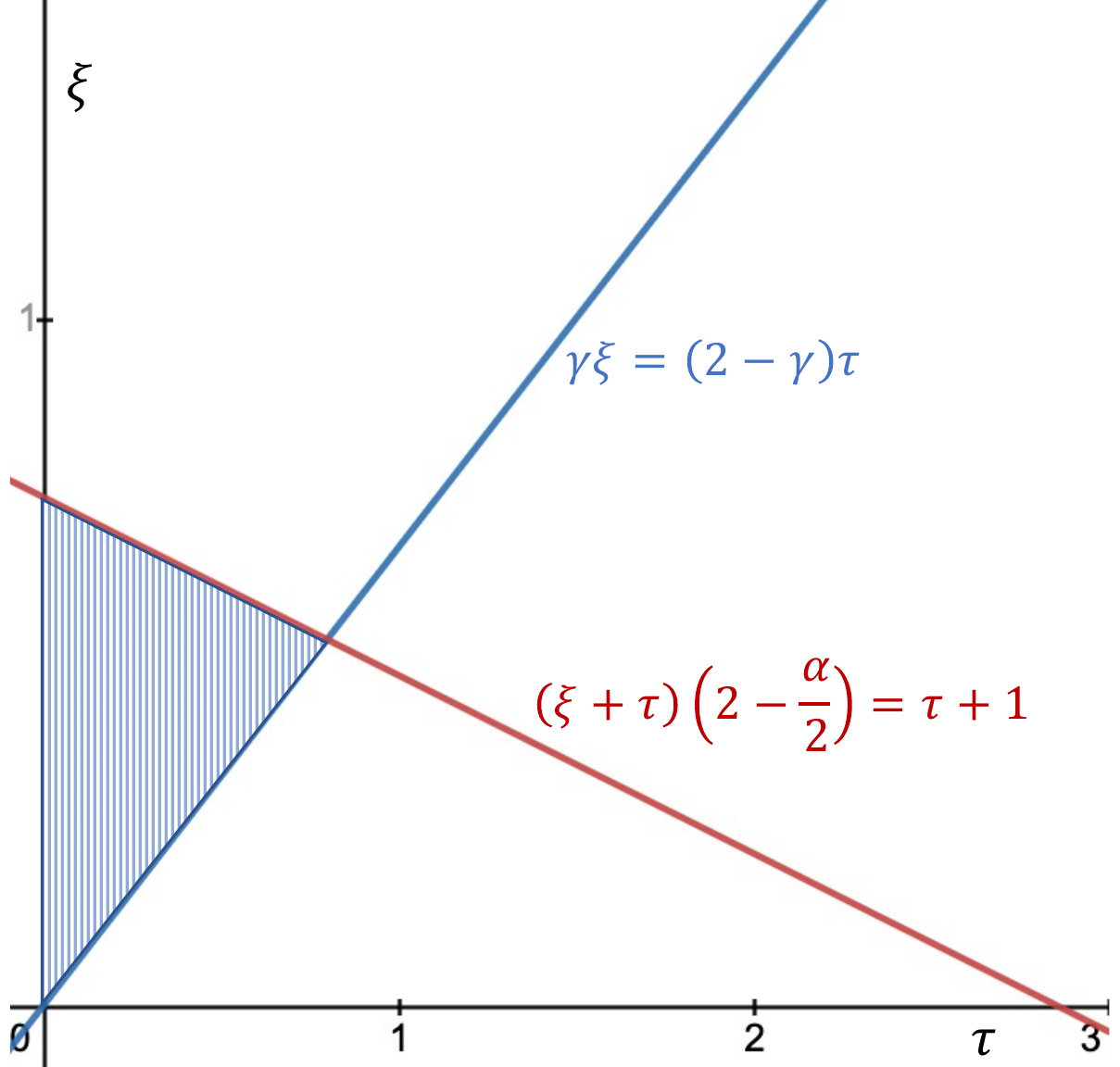}
    \caption{An illustration of the region $(\tau,\zeta)$ that ensures both \eqref{eq:app:main:1} and \eqref{eq:app:main:2} hold. This is the highlighted area.  Within this region, the maximum $\gamma$ is at the red line. In this figure $\gamma=1.2, \alpha=1.3$, and $\beta=1.2/1.3=0.92$.}
    \label{fig:region1}
\end{figure}
Putting everything together, we obtain
\begin{align*}
    &\text{If}\ \tau\leq \frac{\gamma}{4-\alpha-\gamma},\ \text{then}\  \delta_k=\mathcal{O}(k^{-\frac{\alpha(\tau+1)}{4-\alpha}}), \quad \text{with}\quad \eta_k=\Theta(k^{-\frac{\tau+1}{2-\alpha/2}+\tau}).
\end{align*}
Note that $\frac{\gamma}{4-\alpha-\gamma}$ is the intersection point of two lines.
For $\tau>\frac{\gamma}{4-\alpha-\gamma}$, the dynamic is equivalent to 
\begin{align}
         \delta_{t+1}\leq \delta_t+a\eta^2\cdot \delta_t^{\beta}-\eta\mu\delta_t^{\frac{2}{\alpha}},
\end{align}
with the stationary point $r(\eta)=\big(a\eta/\mu\big)^{\frac{\alpha}{2-\gamma}}$. 
Following the steps similar to the previous case, we get the following equation
\begin{align*}
     &1+a\beta\big(\Theta(k^{-\zeta})\big)^2  \Big(\Theta( k^{-\zeta\frac{\alpha}{2-\gamma}})\Big)^{\beta-1}-\frac{2\mu}{\alpha}\Theta(k^{-\zeta})\Big(\Theta( k^{-\zeta\frac{\alpha}{2-\gamma}})\Big)^{2/\alpha-1}=1-\omega_k k^{-1}.
\end{align*}
This leads to $\zeta=\frac{2-\gamma}{4-\alpha-\gamma}$ and subsequently to
\begin{align*}
    &\text{If}\ \tau> \frac{\gamma}{4-\alpha-\gamma},\ \text{then}\  \delta_k=\mathcal{O}(k^{-\frac{\alpha}{4-\alpha-\gamma}}), \quad \text{with}\quad \eta_k=\Theta(k^{-\frac{2-\gamma}{4-\alpha-\gamma}}).
\end{align*}


\end{proof}

\subsection{Proof of Proposition~\ref{lemma_mini-batchsize}}

\textbf{Proposition~\ref{lemma_mini-batchsize}.} 
Let the assumptions of Corollary \ref{corr:main1} hold, $b_k = \Theta\rb{k^{\tau}}$, $T = \Theta\rb{1}$. Then the expected total computational cost (sample complexity) of Algorithm \ref{alg:sgd} is
	 \begin{equation}
	 \text{cost} := T \cdot \sum_{k=0}^{K-1} b_k = \begin{cases} \cO\rb{ \epsilon_f^{-\fr{4-\al}{\al} }} \qquad & \text{for } 0 \leq \tau \leq \frac{\gamma}{4-\alpha-\gamma}, \notag \\
	 \cO\rb{\epsilon_f^{-\frac{\rb{4-\alpha - \gamma} \rb{\tau+1} }{\alpha }} }  \qquad & \text{for } \tau > \frac{\gamma}{4-\alpha-\gamma} . 
	 \end{cases} 
	\end{equation}

\begin{proof}
Corollary~\ref{corr:main1} says that Algorithm \ref{alg:sgd} finds the global $\epsilon_f$-stationary point after $N = K \cdot T$ number of iterations, where
 \begin{align*}
    & N = \mathcal{O}\big(\epsilon_f^{-\frac{4-\alpha}{\alpha(\tau+1)}}\big) \ \text{with}\ \eta_k=\Theta(k^{-\frac{\tau+1}{2-\alpha/2}+\tau})\ \text{if}\  \tau\leq \frac{\gamma}{4-\alpha-\gamma},\\
    & N = \mathcal{O}\big(\epsilon_f^{-\frac{4-\alpha-\gamma}{\alpha}}\big) \ \text{with}\ \eta_k=\Theta(k^{-\frac{2-\gamma}{4-\alpha-\gamma}})\ \text{if}\ \tau> \frac{\gamma}{4-\alpha-\gamma}.
\end{align*}
For $\tau \leq \frac{\gamma}{4-\alpha-\gamma}$, the expected total computational cost (sample complexity) is
$$
 \text{cost} :=  T \cdot \sum_{k=0}^{K-1} b_k = T \sum_{k=0}^{K-1} k^{\tau} =\cO(N^{\tau+1}) = \cO\rb{\epsilon_f^{-\frac{4-\alpha}{\alpha (\tau+1)}\cdot\rb{\tau+1}} } = \cO\rb{\epsilon_f^{-\frac{4-\alpha}{\alpha}}}.
$$
For $\tau > \frac{\gamma}{4-\alpha-\gamma}$, the iteration complexity does not improve and the sample complexity becomes worse when increasing $\tau$
$$
 \text{cost} =  T \cdot \sum_{k=0}^{K-1} b_k = T \sum_{k=0}^{K-1} k^{\tau} =\cO(N^{\tau+1}) = \cO\rb{\epsilon_f^{-\frac{\rb{4-\alpha - \gamma} \rb{\tau+1} }{\alpha }} } .
$$
\end{proof}




\subsection{Proof of Proposition~\ref{pro:lower}}

\textbf{Proposition~\ref{pro:lower}.} \
Consider the following recursion 
\begin{align*}
\delta_{k+1}= \delta_k+a\eta_k^2\cdot h\big(\delta_k\big)-\frac{\eta_k}{2}\phi^2(\delta_k)+ \fr{d\eta_k^2}{b_k},\quad \text{for all } k\geq 0,
\end{align*}
where $a \geq 0$, $d>0$, $h(t)=t^\beta$ with $\beta\in(0,1]$, $\phi(t)=\sqrt{2\mu} t^{1/\alpha}$ with $\alpha\in[1,2]$, and $b_k=\Theta(1)$. 
Then $\delta_k = {\Omega}(k^{-\frac{\alpha}{4-\alpha}})$ for any sequence of $\cb{\eta_k}_{k\geq 0}$. Moreover, this rate is achieved by the choice $\eta_k=\Theta(k^{-\frac{1}{2-\alpha/2}})$.

\begin{proof}

We begin with the fact that if $\delta_k$ defined in \eqref{eq:main_fixed} converges to zero with stepsizes $\{\eta_k\}$, then there exists a $K_0$ such that for all $k\geq K_0$, $\delta_k<1$. 
Hence, for $k\geq K_0$, we have $\delta_k\leq \delta_k^{\beta}$ for $\beta\in(0,1]$.
An immediate consequence of this fact is that for $h(t)=t^\beta$ and $\phi(t)=\sqrt{2\mu} t^{1/\alpha}$, the above dynamic can be bounded as follows

\begin{align}\label{eq:app:tightness2}
     \delta_k+a\eta^2_k \delta_k-\eta_k\mu\delta_k^{\frac{2}{\alpha}}+d\eta^2_k\leq \delta_k+a\eta^2_k \delta_k^\beta-\eta_k\mu\delta_k^{\frac{2}{\alpha}}+d\eta^2_k , \ \forall k\geq K_0.
\end{align}

Let us define two new dynamics as follows, i.e.,

\begin{align}\label{eq:neq:dynamic}
    &r_{k+1}:=r_k+a\eta^2_k r_k-\eta_k\mu r_k^{\frac{2}{\alpha}}+d\eta^2_k,\quad r_0:=\delta_0,\\
    &r_{k+1,\varepsilon}:=r_k\Big(1-a'\eta_k^{1+\frac{2-\alpha-\varepsilon}{2}}\Big)+d'\eta^2_k,\quad r_{0,\varepsilon}:=\delta_0,
\end{align}
First, we show that for any $0<\varepsilon<\frac{2-\alpha}{2}$, there exist $K, a', d'$, such that for all $k\geq K$, $r_{k+1,\varepsilon}\leq r_{k+1}$. 
To do so, we need to understand for what values of $z$, the following inequality holds.
\begin{align*}
  z\Big(1-a'\eta_k^{1+\frac{2-\alpha-\varepsilon}{2}}\Big)+d'\eta^2_k \leq z+a\eta^2_k z-\eta_k\mu z^{\frac{2}{\alpha}}+d\eta^2_k.
\end{align*}
This implies
\begin{align}\label{eq:app:tightness1}
  0 \leq \Big(a'\eta_k^{1+\frac{2-\alpha-\varepsilon}{2}}+a\eta^2_k\Big) z-\eta_k\mu z^{\frac{2}{\alpha}}+(d-d')\eta^2_k.
\end{align}
By choosing $d'=d$, the above inequality holds for 
\begin{align*}
    0\leq z\leq \left(\frac{a'\eta_k^{\frac{2-\alpha-\varepsilon}{2}}+a\eta_k}{\mu}\right)^{\frac{\alpha}{2}}.
\end{align*}
Since $a\geq0$, \eqref{eq:app:tightness1} also holds for 
\begin{align*}
     z\in\Big[0,\ \Big(\frac{a'\eta_k^{\frac{2-\alpha-\varepsilon}{2}}}{\mu}\Big)^{\frac{\alpha}{2}}\Big].
\end{align*}
Therefore, if $r_{k}$ is within the above interval, then $r_{k+1,\varepsilon}\leq r_{k+1}$. 
Using \eqref{eq:app:tightness2}, we know that $r_{k+1}\leq \delta_{k+1}$. 
On the other hand, based on the result of Theorem \ref{lemma_man}, we have $\delta_k=\mathcal{O}(\eta_k^{\frac{\alpha}{2}})$. 
Because of $\frac{2-\alpha-\varepsilon}{2}\leq1$ and the fact that there exists $K$ such that for all $k\geq K$, $\eta_k\leq1$, then $\delta_k$ will lay inside the above interval for large enough $k$.
This implies that there exists $K'$ such that for all $k\geq K'$,  $r_{k+1,\varepsilon}\leq r_{k+1}\leq \delta_{k+1}$.
Finally, using the result of Lemma \ref{lem_ge} with $\epsilon'=\frac{2-\alpha-\varepsilon}{2}$, we obtain the optimal convergence rate of $r_{k,\varepsilon}$ that is 
$$
\Theta\big(k^{-\frac{1-\epsilon'}{1+\epsilon'}}\big)=\Theta\big(k^{-\frac{\alpha-\varepsilon}{4-\alpha-\varepsilon}}\big).
$$
Comparing the above rate with the rate of $\delta_k$ presented in Corollary \ref{corr:main1}, i.e., $\mathcal{O}(k^{-\frac{\alpha}{4-\alpha}})$, concludes the result.  

\end{proof}

Next, we present a generalization of Theorem 3.2 in \cite{gower2019sgd} that helps us to establish our tightness result.
\begin{lemma}\label{lem_ge}
Consider the following recursive equation
\begin{align}\label{eq:lem_ge}
    r_{k+1}:= (1-a'\eta_k^{1+\epsilon'})r_k+c'\eta_k^2,\ k\geq 0,
\end{align}
where $\eta_k\leq\frac{1}{b'}$ for all $k$ and $a',c',\epsilon'\geq0$ with $a'\leq b'$. Then, choosing $s\geq2$ and 
\begin{align*}
    \eta_k:=\begin{cases} 
      \left(\frac{1}{b'}\right)^{\frac{1}{1+\epsilon'}}, & k<[\frac{K}{2}]\ or\ K\leq \frac{b'^{\frac{1-\epsilon'}{1+\epsilon'}}}{a'}, \\
      \left(\frac{2/(1+\epsilon')}{a'(s+k-[\frac{K}{2}])}\right)^{\frac{1}{1+\epsilon'}}, & \text{otherwise}, 
   \end{cases}
\end{align*}
will result in 
$r_K= \Theta \big(K^{-\frac{1-\epsilon'}{1+\epsilon'}}\big)$.
\end{lemma}

\begin{proof}
For $k\leq[\frac{K}{2}]$, we obtain
\begin{align*}
    &r_{k}\leq \left(1-\frac{a'}{b'}\right)^k r_0+\frac{c}{b^{\frac{2}{1+\epsilon'}}}\sum_{t=0}^{k-1}(1-\frac{a'}{b'})^t\leq \left(1-\frac{a'}{b'}\right)^k r_0+d_1, 
\end{align*}
where $d_1:=\frac{c'}{a'b'^{\frac{1-\epsilon'}{1+\epsilon'}}}$.
Note that if $K\leq \frac{b'^{\frac{1-\epsilon'}{1+\epsilon'}}}{a'}$, then 
\begin{align*}
    r_{K}\leq \left(1-\frac{a'}{b'}\right)^{K} r_0+\frac{c'}{a'^2K},
\end{align*}
But for $K> \frac{b'^{\frac{1-\epsilon'}{1+\epsilon'}}}{a'}$ and $k=[K/2]$, we have
\begin{align*}
    r_{[\frac{K}{2}]}\leq \left(1-\frac{a'}{b'}\right)^{[\frac{K}{2}]} r_0+d_1,
\end{align*}
Then for $k\geq1+[\frac{K}{2}]$, we have
\begin{align*}
    &  r_k\leq\left(1-\frac{2/(1+\epsilon')}{s+k-1-[\frac{K}{2}]}\right) r_{k-1}+c'\left(\frac{2/(1+\epsilon')}{a'(s+k-1-[\frac{K}{2}])}\right)^{\frac{2}{1+\epsilon'}}
\end{align*}
Multiplying both sides by $e_k:=(s+k-1-[\frac{K}{2}])^{\frac{2}{1+\epsilon'}}$ results in 
\begin{align}\notag
     e_kr_{k}&\leq \left(s+k-\frac{3+\epsilon'}{1+\epsilon'}-[\frac{K}{2}]\right)\left(s+k-1-[\frac{K}{2}]\right)^{\frac{1-\epsilon'}{1+\epsilon'}}r_{k-1}  +c\left(\frac{2}{a'(1+\epsilon')}\right)^{\frac{2}{1+\epsilon'}}\\ \label{eq:lower-tele}
    &\leq e_{k-1}r_{k-1}+d_2,
\end{align}
where $d_2:=c'\left(\frac{2}{a'(1+\epsilon')}\right)^{\frac{2}{1+\epsilon'}}$. 
The last inequality is due to the Jensen's inequality and the fact that $\log(x)$ is concave, hence, 
\begin{align*}
    \left(x-\frac{2}{1+\epsilon}\right)^{1+\epsilon}x^{1-\epsilon}\leq (x-1)^2.
\end{align*}
Summing up \eqref{eq:lower-tele} from $k=[K/2]+1$ to $k=K$ gives us
\begin{align*}
    e_K r_{K}\leq e_{[K/2]}r_{[K/2]}+d_2(K-[K/2]).
\end{align*}
Consequently, 
\begin{align*}
     &r_{K}\leq \frac{e_{[K/2]}}{e_K}r_{[K/2]}+d_2\frac{(K-[K/2])}{e_K}= \frac{(s-1)^{\frac{2}{1+\epsilon'}}}{e_K}r_{[K/2]}+d_2\frac{(K-[K/2])}{e_K}\\
     &\leq \frac{(s-1)^{\frac{2}{1+\epsilon'}}}{e_K}\left(\left(1-\frac{a'}{b'}\right)^{[\frac{K}{2}]} r_0+d_1\right)+d_2\frac{(K-[K/2])}{e_K}.
\end{align*}
On the other hand, we have $e_K\geq (K-[K/2])^{\frac{2}{1+\epsilon}}\geq (K/2)^{\frac{2}{1+\epsilon}}$, which leads to the following upper bound for $r_K$
\begin{align*}
     &r_{K}\leq \frac{(s-1)^{\frac{2}{1+\epsilon'}}}{(K-[K/2])^{\frac{2}{1+\epsilon'}}}\left(\left(1-\frac{a'}{b'}\right)^{[\frac{K}{2}]} r_0+d_1\right)+\frac{d_2}{(K-[K/2])^{\frac{1-\epsilon'}{1+\epsilon'}}}\\ &\leq\frac{(s-1)^{\frac{2}{1+\epsilon'}}}{(K/2)^{\frac{2}{1+\epsilon'}}}\left(\left(1-\frac{a'}{b'}\right)^{[\frac{K}{2}]} r_0+d_1\right)+\frac{d_2}{(K/2)^{\frac{1-\epsilon'}{1+\epsilon'}}}.
\end{align*}

For the lower bound, we use the following inequality
\begin{align*}
    \left(x-\frac{2}{1+\epsilon}\right)^{1+\epsilon}x^{1-\epsilon}\geq (x-2)^2, \quad \forall x\geq2.
\end{align*}
This implies 
\begin{align*}
    e_kr_k\geq e_{k-2}r_{k-1}+d_2.
\end{align*}
Multiplying the above by $e_{k-1}$, we get
\begin{align*}
    e_{k-1}e_kr_k\geq e_{k-1}e_{k-2}r_{k-1}+d_2e_{k-1}.
\end{align*}
Summing up the above expression from $k=[K/2]+1$ to $k=K$ gives us 
\begin{align*}
    e_{K-1}e_Kr_K\geq e_{[K/2]}e_{[K/2]-1}r_{[K/2]}+d_2\Big(e_{[K/2]}+...+e_{K}\Big).
\end{align*}
Using $\sum_{i=s-1}^{s+K}i^{\frac{2}{1+\epsilon'}}\geq\int_{s-1}^{s+K}x^{\frac{2}{1+\epsilon'}}dx$, we obtain 
\begin{align*}
    r_K=\Omega\Big(\frac{K^{1+\frac{2}{1+\epsilon'}}}{K^{\frac{2}{1+\epsilon'}}K^{\frac{2}{1+\epsilon'}}}\Big)=\Omega(K^{\frac{1-\epsilon'}{1+\epsilon'}}).
\end{align*}
To show that no other designs of stepsizes can achieve better rate, we show that even with the optimal stepsizes, the rate will be the same as above. 
Note that the dynamic in \eqref{eq:lem_ge} is a nonlinear function of the stepsize $\eta_k$ that has a global minimum which can be obtained by taking a derivative of  \eqref{eq:lem_ge} with respect to $\eta_k$. This optimal stepsize is given by 
\begin{align}\label{eq:app:unreal}
    \eta_k=\left(\frac{a'(1+\epsilon) r_k}{2c'}\right)^{1/(1-\epsilon)}.
\end{align}
 Using this stepsizes will lead to the following dynamic
 \begin{align}\label{eq:lem_gem}
     r_{k+1}=r_k(1-Ar_k^{\frac{2}{1-\epsilon}-1}),
 \end{align}
 where $A:=c'(\frac{1-\epsilon}{1+\epsilon})(\frac{a'(1+\epsilon)}{2c'})^{\frac{2}{1-\epsilon}}$. 
 Given the result of Lemma \ref{lemma:app:3}, the convergence rate of this dynamic is $\mathcal{O}(k^{-\frac{1-\epsilon}{1+\epsilon}})$.
See Figure \ref{fig:lower} for an illustration of an example that shows both the simulated $r_k$ in \eqref{eq:lem_gem} and its corresponding optimal rate. Different colours show different $\epsilon$.
\begin{figure}
    \centering
    \includegraphics[scale=.4]{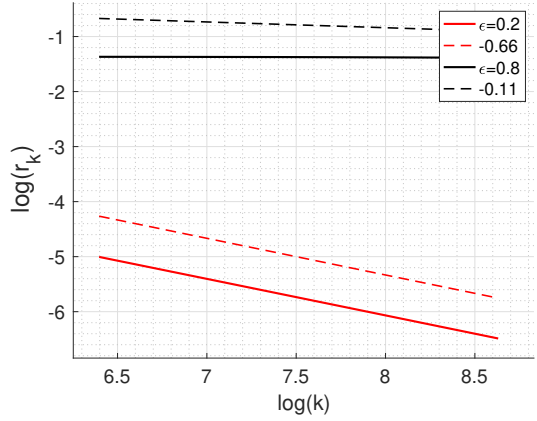}
    \caption{An example to verify equation \eqref{eq:lem_gem} for $\epsilon\in\{0.2,0.8\}$. Solid and dashed lines denote the simulated dynamic in Lemma \ref{lem_ge} and its corresponding theoretical rates, i.e.,  $\mathcal{O}(k^{-\frac{1-\epsilon}{1+\epsilon}})$, respectively. Numbers assigned to dashed lines indicate the slope of those lines.   
    }
    \label{fig:lower}
\end{figure}
 
\end{proof}

\section{Proofs for Section~\ref{sec:variace_reduction} and Additional Discussion}\label{sec:variace_reduction_appendix}

This Section is organized as follows. First, we elaborate on the intuition why one needs to resort to variance reduction techniques in order to improve over \algname{SGD} analysis provided in Section~\ref{sec:main}. Then we highlight the key challenges associated with the analysis of variance reduced methods under global K{\L} condition and introduce a new variance reduced method \algname{PAGER}. We explain the intuition why \algname{PAGER} overcomes the aforementioned challenges and improves over \algname{SGD} in online case \eqref{eq:problem_online}, and over \algname{SGD} and \algname{GD} in finite sum \eqref{eq:problem_finite_sum} case. Finally, we provide convergence guarantees for each setting in Theorems~\ref{thm:PAGE_online_w_restarts1_detailed} and \ref{thm:PAGE_w_restarts_finite_sum_detailed}.\footnote{Note that Theorems~\ref{thm:PAGE_online_w_restarts1_detailed} and \ref{thm:PAGE_w_restarts_finite_sum_detailed} are detailed versions of Theorems~\ref{thm:PAGE_online_w_restarts1} and \ref{thm:PAGE_w_restarts_finite_sum} provided in Section~\ref{sec:variace_reduction}.}

\paragraph{Why \algname{SGD} is not enough?} Notice that the analysis in Section~\ref{sec:main}, in particular, implies that if we want to solve problem \eqref{eq:problem_online} using \algname{SGD} with constant step-size $\eta$ and a mini-batch with replacement gradient estimator of size $b$, we immediately obtain a recursion
\begin{equation}\label{eq:simple_recursion_appendix_1}
    \delta_{t+1}-\delta_t \leq -\eta\mu \delta_t^{\fr{2}{\al}} + \fr{\eta^2 L \sigma^2}{2 b}.
\end{equation}
It is easy to see that if $\eta$ is fixed, then the last (variance) term in the above recursion can be only controlled by selecting large enough $b$.\footnote{The results of Corollary~\ref{corr:main1} and Lemma~\ref{lemma_mini-batchsize} implies that changing $\eta$ and $b$ with iterations does not help. } Assume that we want to solve our problem to $ \epsilon_f$ accuracy ($\delta_T \leq \epsilon_f$). Then to balance the two terms on the RHS, one needs to take $b \sim \epsilon_f^{-\fr{2}{\al}}$.   This choice simplifies the recursion to $\delta_{t+1}-\delta_t \leq - \fr{\eta\mu}{2} \delta_t^{\fr{2}{\al}}$. Applying Lemma~\ref{lemma:app:3} with $c = \fr{2-\al}{\al}$, we conclude that one needs $T \sim \epsilon_f^{\fr{-(2-\al)}{\al}}$ iterations to reach $\delta_T \leq \epsilon_f$. Thus, the total sample complexity is $b \cdot T \sim \epsilon_f^{\fr{-(4-\al)}{\al}}$. This observation implies that we need to construct a more sophisticated gradient estimator than mini-batch estimator in order to improve the sample complexity of \algname{SGD}.

\paragraph{Variance reduction and challenges under K{\L} condition.}
One common technique to design faster algorithms in stochastic optimization is to reduce variance of the gradient estimator using a control variate. It turns out that using such variance reduction techniques one can often design a gradient estimator at a much lower cost, while maintaining the same iteration complexity. Let us turn our attention to one popular variance reduction mechanism called \algname{PAGE}. The main steps of \algname{PAGE} method is described in Section~\ref{sec:variace_reduction}, the detailed pseudo-code is presented in Algorithm~\ref{alg:PAGE}. This method was originally proposed and analyzed for general non-convex and $2$-P{\L} objectives \cite{PAGE}. However, its application to $\al$-P{\L} functions with $\al \in [1,2)$ remains elusive. If we try to apply the standard analysis of \algname{PAGE}, it will become apparent that we face several challenges. In particular, Lemma~\ref{le:PAGE_online} along with Lemma~\ref{le:useful_fact_1} provides the following inequality for the iterates of the Algorithm~\ref{alg:PAGE} 

\begin{eqnarray}\label{eq:simple_recursion_appendix_2}
		\Psi_{t+1} - \Psi_t &\leq&  -  \eta \mu \Psi_t^{\fr{2}{\al}} -   \fr{p_t \lambda_t}{2}  G_t \rb{1 - \fr{4 \eta \mu }{p_t \al } \Psi_t^{\fr{2-\al}{\al}}} + \fr{p_t \lambda_t}{2} \fr{\sigma^2}{b_t} ,
\end{eqnarray}
where $\Psi_t = \delta_t + \lambda G_t$ is a candidate for a Lyapunov function and $G_t$ is the variance of the gradient estimator, and $\lambda > 0$. To illustrate one key obstacle in the analysis of \algname{PAGE} in online setting, let us set $G_t = 0$ for simplicity
\begin{equation}\label{eq:simple_recursion_appendix_3}
\Psi_{t+1} - \Psi_t \leq  -  \eta \mu \Psi_t^{\fr{2}{\al}}   +  \fr{p \lambda \sigma^2}{ 2 b}.
\end{equation}
Now, this recursion is very similar to \eqref{eq:simple_recursion_appendix_1}.  Therefore, the same argument applies here. In particular, one can argue that given constant parameters $\eta$, $b^{\prime}$ and $p$, we need to take $b \sim \epsilon_f^{-\fr{2}{\al}}$. Thus the total sample complexity is again no better than $b \cdot T \sim \epsilon_f^{\fr{-(4-\al)}{\al}}$. Note that the assumption $G_t = 0$ was only made to illustrate one difficulty. Rigorously proving the fact that the term including $G_t$ is small constitutes another challenge. 
\paragraph{Faster rates via \algname{PAGER} in online case.} However, we notice that in \eqref{eq:simple_recursion_appendix_2}, we have one more degree of freedom -- the parameter $p$, which can be selected small enough to ensure smaller per iteration cost of the method. This intuition brings us to \algname{PAGER} (Algorithm~\ref{alg:PAGE_w_restarts}), a new modification of \algname{PAGE} method with varying parameter $p$. \footnote{Note that originally \algname{PAGE} was only analyzed with constant parameter $p$, the extension to an arbitrarily changing $p_t$ is not trivial. } We carefully select the sequences $\cb{p_k}_{k\geq 0}$, $\cb{b_k}_{k\geq 0}$, $\cb{b_k^{\prime}}_{k\geq 0}$ for \algname{PAGER} in order to obtain a small per iteration cost of order $p_k b_k + b_k^{\prime} \sim \epsilon_f^{-1}$. This leads to a much faster convergence with $\epsilon_f^{\nfr{-2}{\al}}$ sample complexity. 
\paragraph{Difficulties in finite sum case and a fix via \algname{PAGER} framework.}
Let us now consider a finite sum problem \eqref{eq:problem_finite_sum} and directly apply Algorithm~\ref{alg:PAGE} with (constant) parameters $\eta$, $p$, $b$, $b^{\prime}$. Then we arrive at the following recursion
\begin{eqnarray}
\Psi_{t+1} - \Psi_t &\leq&  -  \eta \mu \rb{\Psi_t - \lambda G_t}^{\fr{2}{\al}} -   \fr{p \lambda}{2}  G_t   \notag \\
&\leq&  -  \eta \mu \Psi_t^{\fr{2}{\al}} -   \fr{p \lambda}{2}  G_t \rb{1 - \fr{4 \eta \mu (n+1)}{\al} \Psi_t^{\fr{2-\al}{\al}}}, \notag 
\end{eqnarray}
where we applied Lemma~\ref{le:PAGE_online}, \ref{le:useful_fact_1} and selected optimal parameters $p = \fr{1}{n+1}$, $b = n$, $b^{\prime} = 1$. By choosing a small enough stepsize $\eta$, we can unroll the above recursion and obtain the sample complexity $\cO\big(\rb{n \delta_0 + \sqrt{n} \kappa} \rb{\fr{1 + \delta_0}{\epsilon_f}}^{\fr{2-\al}{\al}} \big)$, where $\delta_0 = f(x_0) - f^{\star}$, $\kappa = \nfr{\cL}{\mu}$. However, this complexity is clearly not what one should hope for when analyzing a variance reduction scheme for problem \eqref{eq:problem_finite_sum}. Notably, this complexity can be even worse than the one of standard \algname{GD}, which is $\cO\big( n \kappa  \rb{\fr{1 + \delta_0}{\epsilon_f}}^{\fr{2-\al}{\al}} \big)$, for instance, when $\delta_0 > \kappa$. The main reason for this slowdown is that in the analysis of Algorithm~\ref{alg:PAGE} with constant parameters, we are forced to take small step-sizes of order $\eta = \cO\big( \fr{1}{n\delta_0} \big)$ to ensure progress. Luckily, thanks to a flexible choice of parameters in \algname{PAGER}, we can overcome this difficulty and provide improved convergence guaranties. Specifically, the framework of Algorithm~\ref{alg:PAGE_w_restarts} allows us to select an increasing sequence of step-sizes until it reaches the value $\eta = \cO\big( \fr{1}{\sqrt{n} \cL} \big)$.

\begin{algorithm}[h]
	\caption{\algname{PAGE}}\label{alg:PAGE}
	\begin{algorithmic}[1]
		\STATE Initialization: ${x}_0, {g}_0 \in \R^d$, step-size $\eta$ , number of iterations $T$, probability $p$, batch-sizes $b$, $b^{\prime} $
		\FOR{$t=0, \ldots, T -  1$}
    		\STATE $\xtpo = \xt - \eta g_t$
    		\STATE Sample $\chi \sim \text{Bernoulli}(p)$ 
    		\IF{$\chi=1$}
    		    \STATE $g_{t+1} = \frac{1}{b}\sum_{i=1}^{b} \nabla f_{\xi_{t+1}^i}(x_{t+1})$ 
    		\ELSE
    			\STATE $g_{t+1} = g_{t} +  \frac{1}{b^{\prime}}\sum_{i=1}^{b^{\prime}} \nabla f_{\xi_{t+1}^i}(x_{t+1})  - \frac{1}{b^{\prime}}\sum_{i=1}^{b^{\prime}} \nabla f_{\xi_{t+1}^i}(x_{t})$ 
        	\ENDIF
		\ENDFOR
		\STATE \textbf{Return: ${x}_T$} 
	\end{algorithmic}
\end{algorithm}

\subsection{Proof of Theorem~\ref{thm:PAGE_online_w_restarts1}}

Now we state and prove a detailed version of Theorem~\ref{thm:PAGE_online_w_restarts1}.

\begin{theorem}\label{thm:PAGE_online_w_restarts1_detailed} 
	Let $f(\cdot)$ have the form \eqref{eq:problem_online} and satisfy Assumptions~\ref{ass:l-smooth}, \ref{as:lojasiewicz} (with $\al \in [1, 2)$), \ref{as:UBV} and \ref{as:avg_smoothness_page}, let the sequences in Algorithm~\ref{alg:PAGE_w_restarts} be chosen as 
	\begin{align*}
	    &b_k^{\prime} =  \fr{\al}{8\eta \mu} \rb{ \fr{2^k}{\bar{\Psi}_0}}^{\fr{2-\al}{\al}} , \quad p_k = \fr{1}{1 + b_k^{\prime}}, \\
	    &b_k = \rb{\fr{2 \cdot 2^{\fr{2-\al}{\al}} \cdot 2^k}{\bar{\Psi}_0}}^{\fr{2}{\al}} \fr{\sigma^2}{ 4 \mu \eta^2 \cL^2} , \\
	    &T_k = \fr{2}{\eta \mu} \rb{ 2 \cdot 2^{\fr{2-\al}{\al}} \rb{ \fr{ 2^{k} U}{\bar{\Psi}_0} + 2 \rb{ \fr{\eta \mu}{2} }^{\fr{\al}{2-\al}} }}^{\fr{2-\al}{\al}} ,\\
	    &\eta_k  = \eta  = \fr{1}{\mu} \min\cb{ \fr{1}{2 \kappa} ,  \fr{\al }{8 }  },
	\end{align*}
	where $\bar{\Psi}_0 \eqdef  f(\bar{x}_0)  - f(\xstar) + \lambda_0 \sqnorm{\bar{g}_0 - \nabla f(\bar{x}_0) } $, $ \lambda_0 \eqdef \fr{ b^{\prime}_0 } { 4\eta_0 (1-p_0) \cL^2} $. Then, for any $\epsilon_f > 0$ Algorithm \ref{alg:PAGE_w_restarts} returns a point $\bar{x}_K$ with $\Exp{f(\bar{x}_K) - f^{\star}} \leq \epsilon_f$ after 
$N := \sum_{k=0}^{K-1} T_k =  \cO  \big(\epsilon_f^{-\fr{2-\al}{\al} }\big) $ iterations. The expected total computational cost (sample complexity) is
	 \begin{align*}
	 \text{cost} :=  \sum_{k=0}^{K-1} T_k \rb{p_k b_k + 2 (1-p_k) b_k^{\prime } } = \cO\rb{ \epsilon_f^{-\fr{2}{\al} }} . 
	\end{align*}
\end{theorem}
\begin{proof}
Combining the result of Lemma~\ref{le:PAGE_online} and Lemma~\ref{le:useful_fact_1} with $a = \fr{2}{\al}$, $x = \fr{\lambda G_t}{\Psi_t} \leq 1$, we obtain the following recursion
    \begin{eqnarray}\label{eq:rec_PAGE_online}
		\Psi_{t+1} - \Psi_t &\leq&  -  \eta \mu \Psi_t^{\fr{2}{\al}} -   \fr{p_k \lambda_k}{2}  G_t \rb{1 - \fr{4 \eta \mu }{p_k \al } \Psi_t^{\fr{2-\al}{\al}}} + \fr{p_k \lambda_k}{2} \fr{\sigma^2}{b_k} ,
	\end{eqnarray}
	where ${\Psi}_t \eqdef \delta_t  + \lambda_k G_t $, $G_t \eqdef  \Exp{\fr{1}{2}\sqnorm{g_t - \nfxt}}$, $\delta_t \eqdef \Exp{f(\xt) - f(\xstar)}$,   $ \lambda_k \eqdef \fr{ b^{\prime}_k } { 4\eta_k (1-p_k) \cL^2} $.
	
	Define the sequence $\cb{\bar{\Psi}_k}_{k\geq 0}$ as $\bar{\Psi}_k \eqdef \Exp{ f(\bar{x}_k)  - f(\xstar) + \lambda_k \sqnorm{\bar{g}_k - \nabla f(\bar{x}_k) } }$, which corresponds to the outer loop of the Algorithm~\ref{alg:PAGE_w_restarts}. For each $k = 0,\ldots,K-1$, the inner loop of Algorithm~\ref{alg:PAGE_w_restarts} starts with $x_0$ such that $\Psi_0 \eqdef \bar{\Psi}_k $.
	Let us prove by induction that within the outer loop $\bar{\Psi}_{k} \leq \fr{\bar{\Psi}_{0}}{2^k}$ for $k = 0,\ldots, K-1$ and, for each $k=0,\ldots,K-1 $, within the inner loop we have $\Psi_{t+1} \leq  \Psi_t$ for $t = 0,\ldots T_k - 1$ (unless we reached the desired accuracy $\Psi_t \leq \fr{\bar{\Psi}_k}{ 2 \cdot 2^{\fr{2-\al}{\al}} }  $ within the inner loop). The induction base for the outer loop and $k = 0$ is trivial. The induction base for the inner loop and $t=0$ is verified by the assumption on the step-size and the choice of batch-sizes when $k=0$. Fix $ k = 0,\ldots,K-1$ and $t = 0,\ldots,T_k - 1$ and assume that we have $\Psi_{t} \leq  \Psi_{t-1} \leq \Psi_{0} = \bar{\Psi}_{k}$ and $\bar{\Psi}_{k} \leq \fr{\bar{\Psi}_{0}}{2^k}$. Then it follows from \eqref{eq:rec_PAGE_online} that
	\begin{eqnarray}
		\Psi_{t+1} - \Psi_t &\leq&  -  \eta \mu \Psi_t^{\fr{2}{\al}} -   \fr{p_k \lambda_k}{2}  G_t \rb{1 - \fr{4 \eta \mu }{p_k \al } \Psi_t^{\fr{2-\al}{\al}}} + \fr{p_k \lambda_k}{2} \fr{\sigma^2}{b_k}  \notag \\
		&\leq&  -  \eta \mu \Psi_t^{\fr{2}{\al}} -   \fr{p_k \lambda_k}{2}  G_t \rb{1 - \fr{4 \eta \mu }{p_k \al } \bar{\Psi}_k^{\fr{2-\al}{\al}}} + \fr{p_k \lambda_k}{2} \fr{\sigma^2}{b_k}  \notag \\
		&\leq&  -  \eta \mu \Psi_t^{\fr{2}{\al}} -   \fr{p_k \lambda_k}{2}  G_t \rb{1 - \fr{4 \eta \mu }{p_k \al } \rb{\fr{\bar{\Psi}_0}{2^k}}^{\fr{2-\al}{\al}}} + \fr{p_k \lambda_k}{2} \fr{\sigma^2}{b_k}  \notag \\
		&\overset{(i)}{\leq}&  -  \eta \mu \Psi_t^{\fr{2}{\al}} + \fr{p_k \lambda_k}{2} \fr{\sigma^2}{b_k}  \notag \\
		&\overset{(ii)}{=}&  -  \eta \mu \Psi_t^{\fr{2}{\al}} + \fr{p_k }{2} \fr{\sigma^2}{b_k} \fr{ b_k^{\prime} } { 4\eta (1-p_k) \cL^2} \notag \\
		&\overset{(iii)}{=}&  -  \eta \mu \Psi_t^{\fr{2}{\al}} +  \fr{\sigma^2}{b_k} \fr{ 1 } { 8\eta \cL^2} \notag \\
		&\overset{(iv)}{=}&   -  \eta \mu \Psi_t^{\fr{2}{\al}} +  \fr{\eta \mu}{2} \rb{\fr{\bar{\Psi}_0}{2 \cdot 2^{\fr{2-\al}{\al}} \cdot 2^{k}}}^{\fr{2}{\al}}. \notag
	\end{eqnarray}
	where $(i)$ follows by $p_k \geq \fr{1}{2 b_k^{\prime}} = \fr{4\eta \mu}{\al} \rb{ \fr{\bar{\Psi}_0}{2^k}}^{\fr{2-\al}{\al}}$ and the assumption on the step-size, $(ii)$ is due to $ \lambda_k = \fr{ b_k^{\prime} } { 4\eta (1-p_k) \cL^2} $, $(iii)$ is due to $\fr{p_k b_k^{\prime}}{1-p_k} = 1$, and $(iv)$ holds by the assumption on the batch-size $b_k$. The above recursion guaranties that after at most $T_k = \fr{2}{\eta \mu} \rb{ 2 \cdot 2^{\fr{2-\al}{\al}} \rb{ \fr{ 2^{k} U}{\bar{\Psi}_0} + 2 \rb{ \fr{\eta \mu}{2} }^{\fr{\al}{2-\al}} }}^{\fr{2-\al}{\al}} $ inner loop iterations, we have $\Psi_{T_k} \leq \fr{\Psi_0}{4} = \fr{\bar{\Psi}_k}{ 2 \cdot 2^{\fr{2-\al}{\al}} } = \fr{\bar{\Psi}_0}{ 2 \cdot 2^{\fr{2-\al}{\al}}  \cdot 2^k}$. Indeed, if for $t = 0,\ldots, T_k-1$, we have not reached $\Psi_t \leq \fr{\bar{\Psi}_0}{  2 \cdot 2^{\fr{2-\al}{\al}}  \cdot 2^k}$, then $\Psi_{t+1} - \Psi_t \leq  -  \fr{\eta \mu}{2} \Psi_t^{\fr{2}{\al}} \leq 0$ and by Lemma~\ref{lemma:app:3} (with $c = \fr{2-\al}{\al}$, $b = \nfr{\eta \mu}{2}$), we get $\Psi_{T_k} \leq \fr{\Psi_0}{ 2 \cdot 2^{\fr{2-\al}{\al}} } = \fr{\bar{\Psi}_k}{ 2 \cdot 2^{\fr{2-\al}{\al}} } $. Now it remains to analyze the outer loop of Algorithm~\ref{alg:PAGE_w_restarts}. By the definition of $\bar{\Psi}_k$ and the choice of batch-sizes $b_k^{\prime}$ we have $\lambda_{k+1} \leq 2^{\fr{2-\al}{\al}} \lambda_k$  and  $\bar{\Psi}_{{k+1}} \leq 2^{\fr{2-\al}{\al}} \Psi_{T_k} \leq  \fr{\bar{\Psi}_k}{2} \leq  \fr{\bar{\Psi}_0}{2^{k+1}}$. 
	Thus, the induction step is complete. 
	
	In order to achieve $\bar{\Psi}_K \leq \epsilon_f$, we need $K = \log_2\rb{\fr{\bar{\Psi}_0}{\epsilon_f}}$ outer loop iterations. The total number of iterations is
	\begin{eqnarray}
     N &=& \sum_{k=0}^{K-1} T_k \notag \\
     &=&  \sum_{k=0}^{K-1} \fr{2}{\eta \mu} \rb{ 2 \cdot 2^{\fr{2-\al}{\al}} \rb{ \fr{ 2^{k} U}{\bar{\Psi}_0} + 2 \rb{ \fr{\eta \mu}{2} }^{\fr{\al}{2-\al}} }}^{\fr{2-\al}{\al}}  \notag \\ 
     &=&  \fr{2}{\eta \mu} \rb{2 \cdot 2^{\fr{2-\al}{\al}}}^{\fr{2-\al}{\al}}  \sum_{k=0}^{K-1}  \rb{  \fr{ 2^{k} U}{\bar{\Psi}_0} + 2 \rb{ \fr{\eta \mu}{2} }^{\fr{\al}{2-\al}} }^{\fr{2-\al}{\al}}  \notag \\
     &=&  \fr{2 \cdot 2^{\fr{2(2-\al)}{\al^2}} }{\eta \mu}  \rb{  \fr{ U }{\bar{\Psi}_0} + 2 \rb{ \fr{\eta \mu}{2} }^{\fr{\al}{2-\al}} }^{\fr{2-\al}{\al}}  \sum_{k=0}^{K-1}  \rb{ 2^{\fr{2-\al}{\al}} }^k \notag \\ 
     &\leq&  \fr{2 \cdot 2^{\fr{2(2-\al)}{\al^2}} }{\eta \mu}  \rb{  \fr{ U }{\bar{\Psi}_0} + 2 \rb{ \fr{\eta \mu}{2} }^{\fr{\al}{2-\al}} }^{\fr{2-\al}{\al}}  \rb{ 2^{\fr{2-\al}{\al}} }^{K} \rb{2^{\fr{2-\al}{\al}} - 1}^{-1}  \notag \\ 
     &\leq&  \fr{2 \cdot 2^{\fr{2(2-\al)}{\al^2}} }{\eta \mu}  \rb{  \fr{ U }{\bar{\Psi}_0} + 2 \rb{ \fr{\eta \mu}{2} }^{\fr{\al}{2-\al}} }^{\fr{2-\al}{\al}}  \rb{2^{\fr{2-\al}{\al}} - 1}^{-1}  \rb{ \fr{\bar{\Psi}_0}{\epsilon_f} }^{\fr{2-\al}{\al}} . \notag 
	\end{eqnarray}
	The expected computational cost per iteration is 
	\begin{eqnarray}
	    p_k b_k + 2 (1-p_k) b_k^{\prime }   &\leq &  \fr{b_k}{1 + b_k^{\prime}}  + 2 b_k^{\prime } \notag \\
	     &\leq & \fr{b_k}{b_k^{\prime}} + 2 b_k^{\prime } \notag \\
	      &\leq & \fr{\rb{\fr{2 \cdot 2^{\fr{2-\al}{\al}} \cdot 2^k}{\bar{\Psi}_0}}^{\fr{2}{\al}} \fr{\sigma^2}{ 4 \mu \eta^2 \cL^2}}{ \fr{\al}{8\eta \mu} \rb{ \fr{2^k}{\bar{\Psi}_0}}^{\fr{2-\al}{\al}} } + 2  \fr{\al}{8\eta \mu}  \rb{ \fr{2^k}{\bar{\Psi}_0}}^{\fr{2-\al}{\al}}   \notag \\
	      &\leq &  \rb{2 \cdot 2^{\fr{2-\al}{\al}} }^{\fr{2}{\al}}  \fr{ 2 \sigma^2}{  \eta \cL^2}  \fr{2^k}{\bar{\Psi}_0}  +   \fr{\al}{4\eta \mu}   \rb{ \fr{2^k}{\bar{\Psi}_0}}^{\fr{2-\al}{\al}}  \notag \\
	       &\leq &   \fr{2 \sigma^2 \cdot 2^{\nfr{4}{\al^2}}  }{ 4  \eta \cL^2}  \fr{2^k}{\bar{\Psi}_0}  +   \fr{\al}{4\eta \mu}   \rb{ \fr{2^k}{\bar{\Psi}_0}}^{\fr{2-\al}{\al}}   \notag \\
	   &\leq &  \rb{ \fr{\sigma^2 \cdot 2^{\nfr{4}{\al^2}}  }{ 4 \eta \cL^2 \bar{\Psi}_0 }    + \fr{\al}{4\eta \mu \bar{\Psi}_0^{\fr{2-\al}{\al} }}  } {2^k}.   \notag
	\end{eqnarray}
	Denote $A := \rb{ \fr{\sigma^2 \cdot 2^{\nfr{4}{\al^2}}  }{ 4 \eta \cL^2 \bar{\Psi}_0 }    + \fr{\al}{4\eta \mu \bar{\Psi}_0^{\fr{2-\al}{\al} }}  }$, then the total cost is 
	\begin{eqnarray}
	 \text{cost} &=&  \sum_{k=0}^{K-1} T_k \rb{p_k b_k + 2 (1-p_k) b_k^{\prime } }   \notag \\
	 &=& A \sum_{k=0}^{K-1} T_k \cdot {2^k} \notag \\
	 &=& A \sum_{k=0}^{K-1}  \fr{2}{\eta \mu} \rb{ 2 \cdot 2^{\fr{2-\al}{\al}} \rb{ \fr{ 2^{k} U}{\bar{\Psi}_0} + 2 \rb{ \fr{\eta \mu}{2} }^{\fr{\al}{2-\al}} }}^{\fr{2-\al}{\al}}   2^k \notag \\
	 &=& {2 \cdot 2^{\fr{2(2-\al)}{\al^2}} A }  \rb{  \fr{2}{\eta \mu} \rb{ \fr{ U }{\bar{\Psi}_0} }^{\fr{2-\al}{\al}} \sum_{k=0}^{K-1}   \rb{2^k}^{\fr{2-\al}{\al}} 2^k  + 2^{\fr{2-\al}{\al}}  \sum_{k=0}^{K-1}   2^k  }   \notag \\
	 &=& {2 \cdot 2^{\fr{2(2-\al)}{\al^2}} A }  \rb{  \fr{2}{\eta \mu} \rb{ \fr{ U }{\bar{\Psi}_0} }^{\fr{2-\al}{\al}} \sum_{k=0}^{K-1}   \rb{2^k}^{\fr{2}{\al}}   + 2^{\fr{2-\al}{\al}}  \sum_{k=0}^{K-1}   2^k  }   \notag \\
	  &=& {2 \cdot 2^{\fr{2(2-\al)}{\al^2}} A }  \rb{  \fr{2}{\eta \mu} \rb{ \fr{ U }{\bar{\Psi}_0} }^{\fr{2-\al}{\al}} \rb{2^{K}}^{\fr{2}{\al}}  \rb{2^{\nfr{2}{\al} } - 1 }^{-1}   + 2^{\fr{2-\al}{\al}}  2^K }   \notag ,
	  \end{eqnarray}
	  which further simplifies by using the value of $A$ and the step-size
	  	\begin{eqnarray}
	 \text{cost}
	 &=& \cO\rb{  \fr{A}{\eta \mu }  \rb{  \fr{ 1 }{\bar{\Psi}_0}  }^{\fr{2-\al}{\al}}     \rb{2^{K}}^{\fr{2}{\al}}   }  \notag \\
	 &=& \cO\rb{  \fr{A}{\eta \mu }  \rb{  \fr{ 1 }{\bar{\Psi}_0}  }^{\fr{2-\al}{\al}}    \rb{\fr{\bar{\Psi}_0}{\epsilon_f}}^{\fr{2}{\al}}   }  \notag \\
	 &=& \cO\rb{  \fr{A \bar{\Psi}_0 }{\eta \mu }     \rb{\fr{1}{\epsilon_f}}^{\fr{2}{\al} }    }  \notag \\
	 &=& \cO\rb{  \rb{ \fr{\sigma^2 }{\mu} + \fr{\bar{\Psi}_0^{\fr{2(\al-1)}{\al}}}{\eta^2\mu^2}  } \rb{\fr{1}{\epsilon_f}}^{\fr{2}{\al} }    }  \notag \\
	 &=& \cO\rb{  \rb{ \fr{\sigma^2 }{\mu} + \kappa^2 \bar{\Psi}_0^{\fr{2(\al-1)}{\al}}  } \rb{\fr{1}{\epsilon_f}}^{\fr{2}{\al} }    }  \notag \\
	 &=& \cO\rb{  \epsilon_f^{\nfr{-2}{\al} }   } . \notag 
	\end{eqnarray}
\end{proof}

\subsection{Proof of Theorem~\ref{thm:PAGE_w_restarts_finite_sum}}

Now we state and prove a detailed version of Theorem~\ref{thm:PAGE_w_restarts_finite_sum}.

\begin{theorem}\label{thm:PAGE_w_restarts_finite_sum_detailed}
    Let $f(\cdot)$ have the form \eqref{eq:problem_finite_sum} and satisfy Assumptions \ref{ass:l-smooth}, \ref{as:lojasiewicz} (with $\al \in [1, 2)$) and \ref{as:avg_smoothness_page}, let the sequences in Algorithm~\ref{alg:PAGE_w_restarts} be chosen as $p_k = \fr{1}{n+1}$, $b_k^{\prime} = 1$, $b_k = n$, 
    \begin{align*}
    &T_k =  \fr{1}{\eta_k \mu } \rb{\fr{U 2^{k+1}}{\bar{\Psi}_0 } + 2 \rb{\eta_k \mu}^{\fr{\al}{2-\al}}  }^{\fr{2-\al}{\al}},\\
    &\eta_k = \min\cb{ \fr{1}{2 \sqrt{n}\cL}  , \fr{\al}{4\mu (n+1)}\rb{ \fr{2^k}{\bar{\Psi}_0} }^{\fr{2-\al}{\al}} }, \end{align*}
    where $\bar{\Psi}_0 \eqdef  f(\bar{x}_0)  - f(\xstar) + \lambda_0 \sqnorm{\bar{g}_0 - \nabla f(\bar{x}_0) } $, $ \lambda_0 \eqdef \fr{ b^{\prime} } { 4\eta_0 (1-p) \cL^2} $
    Then, for any $\epsilon_f > 0$, Algorithm \ref{alg:PAGE_w_restarts} returns a point $\bar{x}_K$ with $\Exp{f(\bar{x}_K) - f^{\star}} \leq \epsilon_f$ after 
    \begin{eqnarray}
    N &:=& \sum_{k=0}^{K-1} T_k =  \cwO\rb{  n + \sqrt{n} \kappa  \epsilon_f^{-\fr{2-\al}{\al} } } \notag
    \end{eqnarray}
    iterations. 
    The expected total computational cost (sample complexity) is
	 \begin{eqnarray}
	 \text{cost} :=  \sum_{k=0}^{K-1} T_k \rb{p_k b_k + 2 (1-p_k) b_k^{\prime } } = \cwO\rb{ n + \sqrt{n} \kappa \epsilon_f^{-\fr{2-\al}{\al}}  } . \notag
	\end{eqnarray}
\end{theorem}

\begin{proof}
Combining the result of Lemma~\ref{le:PAGE_online} and Lemma~\ref{le:useful_fact_1} with $a = \fr{2}{\al}$, $x = \fr{\lambda G_t}{\Psi_t} \leq 1$ and noticing that $\sigma^2 = 0$, we obtain the following recursion

\begin{eqnarray}\label{eq:finite_sum_PAGE_main}
		\Psi_{t+1} - \Psi_t &\leq&  -  \eta \mu \Psi_t^{\fr{2}{\al}} -   \fr{p \lambda}{2}  G_t \rb{1 - \fr{4 \eta \mu (n+1)}{\al} \Psi_t^{\fr{2-\al}{\al}}},
\end{eqnarray}
where ${\Psi}_t \eqdef \delta_t  + \lambda_k G_t $, $G_t \eqdef  \Exp{\fr{1}{2}\sqnorm{g_t - \nfxt}}$, $\delta_t \eqdef \Exp{f(\xt) - f(\xstar)}$,   $ \lambda_k \eqdef \fr{ b^{\prime} } { 4\eta_k (1-p) \cL^2} $.

Define the sequence $\cb{\bar{\Psi}_k}_{k\geq 0}$ as $\bar{\Psi}_k \eqdef \Exp{ f(\bar{x}_k)  - f(\xstar) + \lambda_k \sqnorm{\bar{g}_k - \nabla f(\bar{x}_k) } }$ and $ \lambda_k \eqdef \fr{ b^{\prime} } { 4\eta_k (1-p) \cL^2} $, which corresponds to the outer loop of the algorithm. For each $k = 0,\ldots,K-1$, the inner loop of Algorithm~\ref{alg:PAGE_w_restarts} starts with $x_0$ such that $\Psi_0 \eqdef \bar{\Psi}_k $. Let us prove by induction that the sequence $\cb{\bar{\Psi}_k}_{k\geq 0}$ satisfies $\bar{\Psi}_k \leq \fr{\bar{\Psi}_0}{2^k}$ for all $k = 0,\ldots,K-1$. The induction base for $k = 0$ is trivial. Let us prove the induction step for $k+1$. The evolution of the inner loop is characterized by \eqref{eq:online_PAGE_main} and given the assumption on the step-size, we have $\Psi_{t+1} - \Psi_t \leq  -  \eta \mu \Psi_t^{\fr{2}{\al}} $ for all $t = 0, \ldots, T_k-1$ .  Therefore, by Lemma~\ref{lemma:app:3} (with $c = \fr{2-\al}{\al}$, $b = \eta \mu$) we have  
 \begin{align*}
     \Psi_{T_k} &\leq \fr{ U + \rb{\eta_k \mu}^{\fr{\al}{2-\al}}  \bar{\Psi}_k }{ \rb{ \eta_k \mu T_k }^{\fr{\al}{2-\al}} } = \fr{ U + \rb{\eta_k \mu}^{\fr{\al}{2-\al}} \bar{\Psi}_k }{ \fr{U \cdot 2^{k+1}}{\bar{\Psi}_0}  + 2 \rb{\eta_k \mu}^{\fr{\al}{2-\al}} } \notag \\
     & =  \fr{ U + \rb{\eta_k \mu}^{\fr{\al}{2-\al}}\bar{\Psi}_k }{ U  +  \rb{\eta_k \mu}^{\fr{\al}{2-\al}} \fr{\bar{\Psi}_0}{2^{k}} }  \cdot \fr{\bar{\Psi}_0}{2^{k+1}} \overset{(i)}{\leq} \fr{\bar{\Psi}_0}{2^{k+1}},
 \end{align*}
 where in $(i)$, we used $\bar{\Psi}_k \leq \fr{\bar{\Psi}_0}{2^k}$. Furthermore, since $\eta_{k+1} \geq \eta_k$,
 then $\lambda_{k+1} \leq \lambda_k$  and  $\bar{\Psi}_{{k+1}} \leq \Psi_{T_k} \leq  \fr{\bar{\Psi}_0}{2^{k+1}} $, and the induction step is complete. 

In order to achieve $\bar{\Psi}_K \leq \epsilon_f$, we need $K = \log_2\rb{\fr{\bar{\Psi}_0}{\epsilon_f}}$ outer loop iterations. The total number of iterations is 
 \begin{eqnarray}
 N &=& \sum_{k=0}^{K-1} T_k \notag \\
 &\overset{(i)}{\leq} &  \sum_{k=0}^{K-1} \max\cb{ \fr{4 (n+1)}{\al} \rb{\fr{\bar{\Psi}_0}{2^k}}^{\fr{(2-\al)}{\al}}, 2\sqrt{n} \kappa}  \rb{\fr{U 2^{k+1}}{\bar{\Psi}_0 } + \fr{\mu }{\sqrt{n} \cL} }^{\fr{2-\al}{\al}}  \notag \\
 &\leq &  \sum_{k=0}^{K-1} \max\cb{ \fr{4 (n+1)}{\al}   \rb{2 \rb{U + \fr{\bar{\Psi}_0}{\sqrt{n} \kappa}  } }^{\fr{2-\al}{\al}} , 2\sqrt{n} \kappa \rb{\fr{2U}{\bar{\Psi}_0} + \fr{1}{\sqrt{n} \kappa}}^{\fr{2-\al}{\al}} \rb{ 2^{\fr{2-\al}{\al}} }^{k}  } \notag \\
 &\leq &   \max\cb{ \fr{4 (n+1)}{\al}   \rb{2 \rb{U + \fr{\bar{\Psi}_0}{\sqrt{n} \kappa} } }^{\fr{2-\al}{\al}} K , 2\sqrt{n} \kappa \rb{\fr{2U}{\bar{\Psi}_0} + \fr{1}{\sqrt{n} \kappa}}^{\fr{2-\al}{\al}} \rb{ 2^{\fr{2-\al}{\al}} }^{K} \rb{2^{\fr{2-\al}{\al}} - 1}^{-1}  } \notag \\
 &\leq &   \max\cb{ \fr{4 (n+1)}{\al}   \rb{2 \rb{U + \fr{\bar{\Psi}_0}{\sqrt{n} \kappa} } }^{\fr{2-\al}{\al}} \log_2\rb{\fr{\bar{\Psi}_0}{\epsilon_f}} , \fr{2\sqrt{n}}{2^{\fr{2-\al}{\al}} - 1} \kappa \rb{\fr{2U}{\bar{\Psi}_0} + \fr{1}{\sqrt{n} \kappa}}^{\fr{2-\al}{\al}} \rb{ \fr{\bar{\Psi}_0}{\epsilon_f} }^{\fr{2-\al}{\al}}  } \notag \\
  &= &   \cwO\rb{ n + \sqrt{n} \kappa  \epsilon_f^{-\fr{2-\al}{\al} } } , \notag 
	\end{eqnarray}
	where in $(i)$ we used the assumption on the step-sizes. 
The expected computational cost per iteration is $ {p_k b_k + 2 (1-p_k) b_k^{\prime } }  \leq 3$ and thus the total cost is $\cwO\big( n + \sqrt{n} \kappa \epsilon_f^{-\fr{2-\al}{\al}} \big) $.
\end{proof}

\subsection{Technical lemmas}
	\begin{lemma}\label{le:useful_fact_1}
		Let $x \leq 1$ and $a \geq 1$, then $(1-x)^a \geq 1 - a x $.
	\end{lemma}
\begin{proof}
    The results follows directly by applying the definition of convexity. 
\end{proof}
The following lemma is standard, we provide its proof for completeness. 
\begin{lemma}\label{le:aux_smooth_lemma}
	Suppose that function $f(\cdot)$ is $L$-smooth and let $x_{t+1}\eqdef x_{t}-\eta g_{t} ,$ where $g_t\in \R^d$ is any vector, and $\eta>0$ any scalar. Then we have
	\begin{eqnarray}\label{eq:aux_smooth_lemma}
		f(x_{t+1}) \leq f(x_{t})-\fr{\eta}{2}\sqnorm{\nabla f(x_{t})}-\left(\fr{1}{2 \eta}-\fr{L}{2}\right)\sqnorm{x_{t+1}-x_{t}}+\fr{\eta}{2}\sqnorm{g_{t}-\nabla f(x_{t})}.
	\end{eqnarray}
\end{lemma}

\begin{proof}
 Define $\bar{x}_{t+1}:=x_{t}-\eta \nabla f\left(x_{t}\right)$, then using Assumption~\ref{ass:l-smooth} after some rearrangements we obtain
\begin{eqnarray*}
f\left(x_{t+1}\right) & {\leq}& f\left(x_{t}\right)+\left\langle\nabla f\left(x_{t}\right), x_{t+1}-x_{t}\right\rangle+\frac{L}{2}\left\|x_{t+1}-x_{t}\right\|^{2} \\
&=& f\left(x_{t}\right)+\left\langle\nabla f\left(x_{t}\right)-g_{t}, x_{t+1}-x_{t}\right\rangle+\left\langle g_{t}, x_{t+1}-x_{t}\right\rangle+\frac{L}{2}\left\|x_{t+1}-x_{t}\right\|^{2} \\
&=&f\left(x_{t}\right)+\left\langle\nabla f\left(x_{t}\right)-g_{t},-\eta g_{t}\right\rangle-\left(\frac{1}{\eta}-\frac{L}{2}\right)\left\|x_{t+1}-x_{t}\right\|^{2}\\
&=& f\left(x_{t}\right)+\eta\left\|\nabla f\left(x_{t}\right)-g_{t}\right\|^{2}-\eta\left\langle\nabla f\left(x_{t}\right)-g_{t}, \nabla f\left(x_{t}\right)\right\rangle-\left(\frac{1}{\eta}-\frac{L}{2}\right)\left\|x_{t+1}-x_{t}\right\|^{2} \\
&=& f\left(x_{t}\right)+\eta\left\|\nabla f\left(x_{t}\right)-g_{t}\right\|^{2}-\frac{1}{\eta}\left\langle x_{t+1}-\bar{x}_{t+1}, x_{t}-\bar{x}_{t+1}\right\rangle-\left(\frac{1}{\eta}-\frac{L}{2}\right)\left\|x_{t+1}-x_{t}\right\|^{2} \\
&=& f\left(x_{t}\right)+\eta\left\|\nabla f\left(x_{t}\right)-g_{t}\right\|^{2}-\left(\frac{1}{\eta}-\frac{L}{2}\right)\left\|x_{t+1}-x_{t}\right\|^{2} \\
&& \qquad - \frac{1}{2 \eta}\left(\left\|x_{t+1}-\bar{x}_{t+1}\right\|^{2}+\left\|x_{t}-\bar{x}_{t+1}\right\|^{2}-\left\|x_{t+1}-x_{t}\right\|^{2}\right) \\
&=& f\left(x_{t}\right)+\eta\left\|\nabla f\left(x_{t}\right)-g_{t}\right\|^{2}-\left(\frac{1}{\eta}-\frac{L}{2}\right)\left\|x_{t+1}-x_{t}\right\|^{2} \\
&& \qquad - \frac{1}{2 \eta}\left(\eta^{2}\left\|\nabla f\left(x_{t}\right)-g_{t}\right\|^{2}+\eta^{2}\left\|\nabla f\left(x_{t}\right)\right\|^{2}-\left\|x_{t+1}-x_{t}\right\|^{2}\right) \\
&=& f\left(x_{t}\right)-\frac{\eta}{2}\left\|\nabla f\left(x_{t}\right)\right\|^{2}-\left(\frac{1}{2 \eta}-\frac{L}{2}\right)\left\|x_{t+1}-x_{t}\right\|^{2}+\frac{\eta}{2}\left\|g_{t}-\nabla f\left(x_{t}\right)\right\|^{2}.
\end{eqnarray*}

\end{proof}

\begin{lemma}\label{lemma:app:3}
Let  $\cb{r_k}_{k\geq 0}$ be a non-negative sequence, which satisfies
\begin{align*}
    r_{k+1} \leq r_k(1-b r_k^{c}),\quad \text{for all } k
\end{align*}
and $c > 0$. Then 
\begin{eqnarray}
 r_k \leq \fr{U + b^{\nfr{1}{c}} r_0}{\rb{ b \rb{k + 1} }^{\nfr{1}{c}} },\notag
\end{eqnarray}
where $U \eqdef 2^{\nfr{1}{c}} \cdot  c^{-\fr{2}{c} -1 } + c^{-\nfr{1}{c}}$. 
\end{lemma}

\begin{proof}
	Define $u_k \eqdef \varphi(k) r_k $, $\varphi(k) \eqdef  \rb{ b (k+1)}^{\nfr{1}{c}} $. Then using $\varphi(k+1) - \varphi(k) \leq \fr{1}{c} \fr{\varphi(k+1)}{k+2}$ and $1\leq \varphi(k+1) \rb{ \varphi(k) }^{-1}  \leq 2^{\nfr{1}{c}}$, we obtain
	\begin{eqnarray}  
		u_{k+1} - u_k &=& \varphi(k+1) r_{k+1} -\varphi(k) r_k \notag \\
		&\leq & \rb{ \varphi(k+1) -\varphi(k)} r_k - b \varphi(k+1) r_k^{1+c}\notag \\
		&= &   \rb{ \varphi(k+1) -\varphi(k)} \rb{ \varphi(k) }^{-1}  u_k - b \varphi(k+1) \rb{\varphi(k)}^{-1-c}  u_k^{1+c}  \notag \\
		&= &   \rb{ \varphi(k+1) -\varphi(k)} \rb{ \varphi(k) }^{-1}  u_k \rb{ 1 - \fr{\varphi(k+1) u_k^{c} }{(k+1) \rb{\varphi(k+1) - \varphi(k)}} }\notag \\
		& \leq &   \rb{ \varphi(k+1) -\varphi(k)} \rb{ \varphi(k) }^{-1}  u_k \rb{ 1 - c u_k^{c}  } .\notag 
	\end{eqnarray}
	It follows from the above recursion that the sequence $\cb{u_k}_{k\geq 0}$ is bounded for all $k$. Indeed, define $F(k, u) \eqdef \rb{ \varphi(k+1) -\varphi(k)} \rb{ \varphi(k) }^{-1}  u \rb{ 1 - c u^{c}  }$. Notice that for all $k\geq 0$ and $u > c^{-\nfr{1}{c}}$ we have $F(k, u) < 0$ and for all $k, u \geq 0$ we have $F(k,u) \leq 2^{\nfr{1}{c}} \cdot  c^{-\fr{2}{c} -1 }$. Now it is straightforward to see that $u_k \leq u_0 + 2^{\nfr{1}{c}} \cdot  c^{-\fr{2}{c} -1 } +c^{-\nfr{1}{c}}$. It only remains to return to $r_k$ sequence to obtain the desired result. 
\end{proof}
\begin{lemma}[Lemma 4 of \citep{PAGE}]\label{le:rec_page_online}
	Let Assumptions~\ref{as:UBV} and \ref{as:avg_smoothness_page} hold, and let for $\chi \sim \text{Bernoulli}(p)$ and $g_t \in \R^d$, we construct $g_{t+1}$ via 
	\begin{equation}
	g_{t+1} = \begin{cases}
				\frac{1}{b}\sum_{i=1}^{b} \nabla f_{\xi_{t+1}^i}(x_{t+1})   & \text{if} \quad\  \chi = 1,\\
				g_{t} +  \frac{1}{b^{\prime}}\sum_{i=1}^{b^{\prime}} \rb{ \nabla f_{\xi_{t+1}^i}(x_{t+1}) - \nabla f_{\xi_{t+1}^i}(x_{t}) }  &\text{if} \quad\   \chi = 0 .
			\end{cases}
	\end{equation}
	Then 
	\begin{equation}\label{eq:rec_page_online}
		G_{t+1} - G_t\leq - p { G_t} + \fr{ (1-p) \cL^2}{b^{\prime}} R_t + \fr{p\sigma^2}{2 b },
	\end{equation}
	where $G_t \eqdef  \Exp{\fr{1}{2}\sqnorm{g_t - \nfxt}}$, $R_n \eqdef \Exp{\fr{1}{2}\sqnorm{\xtpo - \xt} }$.
\end{lemma}

\begin{proof}

\begin{eqnarray*}
G_{t+1} &=&\mathbb{E}\left[\fr{1}{2}\left\|g_{t+1}-\nabla f\left(x_{t+1}\right)\right\|^{2}\right]  \\
&=& p \mathbb{E}\left[\fr{1}{2}\left\|\frac{1}{b}\sum_{i=1}^{b} \nabla f_{\xi_{t+1}^i}(x_{t+1}) - \nabla f\left(x_{t+1}\right)\right\|^{2}\right] \\
&& \qquad + \left(1-p\right) \Exp{\fr{1}{2}\sqnorm{	g_{t} +  \frac{1}{b^{\prime}}\sum_{i=1}^{b^{\prime}} \rb{ \nabla f_{\xi_{t+1}^i}(x_{t+1}) - \nabla f_{\xi_{t+1}^i}(x_{t}) } - \nabla f\left(x_{t+1}\right)}}  \\
&\leq& \frac{p \sigma^{2}}{2b} + \left(1-p\right) \Exp{\fr{1}{2}\sqnorm{	g_{t} +  \frac{1}{b^{\prime}}\sum_{i=1}^{b^{\prime}} \rb{ \nabla f_{\xi_{t+1}^i}(x_{t+1}) - \nabla f_{\xi_{t+1}^i}(x_{t}) } - \nabla f\left(x_{t+1}\right)}} \\
&=&\ \frac{p \sigma^{2}}{2b} +\left(1-p\right)  \Exp{\fr{1}{2}\sqnorm{ g_{t}-\nabla f\left(x_{t}\right)+\widetilde{\Delta}(x_{t+1}, x_t) - \Delta(x_{t+1}, x_t) }}  \\
&=& \frac{p \sigma^{2}}{2b} + \left(1-p\right)\Exp{\fr{1}{2}\sqnorm{ g_{t}-\nabla f\left(x_{t}\right) } }  +\left(1-p\right)  \Exp{\fr{1}{2}\sqnorm{\widetilde{\Delta}(x_{t+1}, x_t) - \Delta(x_{t+1}, x_t) }} \\
&\leq&  \left(1-p\right)\Exp{\fr{1}{2}\sqnorm{ g_{t}-\nabla f\left(x_{t}\right) } }   + \fr{ (1-p) \cL^2}{b^{\prime}} \Exp{\fr{1}{2}\sqnorm{x_{t+1} - x_t}} + \frac{p \sigma^{2}}{2b}\\
&=&  \left(1-p\right)G_t   + \fr{ (1-p) \cL^2}{b^{\prime}} R_t + \frac{p \sigma^{2}}{2b},
\end{eqnarray*}
where the first inequality holds by Assumption~\ref{as:UBV} and the second inequality is due to Assumption~\ref{as:avg_smoothness_page} with $\widetilde{\Delta}(x, y) \eqdef  \frac{1}{b^{\prime}}\sum_{i=1}^{b^{\prime}} \rb{ \nabla f_{\xi_{t+1}^i}(x) - \nabla f_{\xi_{t+1}^i}(y) }  $, $\Delta(x,y) \eqdef \nabla f(x) - \nabla f(y)$, $x = x_{t+1}$, $y = x_t$.

\end{proof}

\begin{lemma}\label{le:PAGE_online}
	Let $f(\cdot)$ satisfy Assumptions~\ref{ass:l-smooth}, \ref{as:lojasiewicz}, \ref{as:UBV} and \ref{as:avg_smoothness_page}. Assume that the step-size in Algorithm~\ref{alg:PAGE} satisfies
	\begin{equation}\label{eq:PAGE_sz_online}
		\eta \leq \min\cb{ \fr{1}{2 L} ,  \sqrt{\fr{p b^{\prime}}{1-p}} \fr{1}{2\cL}  }.
	\end{equation}
   Define $\Psi_t \eqdef \Exp{ f(\xt)  - f(\xstar) + \lambda \sqnorm{g_t - \nfxt} }$, $ \lambda \eqdef \fr{ b^{\prime} } { 4\eta (1-p) \cL^2} $. Then Algorithm~\ref{alg:PAGE} generates a sequence of points $\cb{x_t}_{t\geq 0}$ such that
   	\begin{eqnarray}\label{eq:online_PAGE_main}
		\Psi_{t+1} - \Psi_t &\leq&  -  \eta \mu \rb{\Psi_t - \lambda G_t}^{\fr{2}{\al}} -   \fr{p \lambda}{2}  G_t  +  \fr{p \lambda}{2} \fr{\sigma^2}{b}.  
	\end{eqnarray}
\end{lemma}
\begin{proof}
	Using the notation $G_t \eqdef  \Exp{\fr{1}{2}\sqnorm{g_t - \nfxt}}$, $R_t \eqdef \Exp{\fr{1}{2}\sqnorm{\xtpo - \xt} }$, $\delta_t \eqdef \Exp{f(\xt) - f(\xstar)}$ and assumption on the step-size $\eta \leq \fr{1}{2L}$, it follows by Lemma~\ref{le:aux_smooth_lemma} that 
	\begin{eqnarray}
		\delta_{t+1} - \delta_t \leq  -  \fr{\eta}{2}\Exp{\sqnorm{\nabla f(x_{t})}} - \fr{1}{2\eta} R_t+{\eta} G_t. \notag
	\end{eqnarray}
	
	Using Assumption~\ref{as:lojasiewicz}, Jensen's inequality for $x \mapsto x^{\fr{2}{\al}}$, we get 
	\begin{eqnarray}
		\delta_{t+1} - \delta_t \leq  -  \eta \mu \delta_t^{\fr{2}{\al}} + {\eta} G_t - \fr{1}{2\eta} R_t . \notag
	\end{eqnarray}
	For $p < 1$, it follows from Lemma~\ref{le:rec_page_online} that
	\begin{equation}
		- R_t  \leq - \fr{b^{\prime}} { (1-p) \cL^2} \rb{G_{t+1} - G_{t}} - \fr{p b^{\prime}} { (1-p) \cL^2}  { G_t} + \fr{b^{\prime}} { (1-p) \cL^2} \fr{p\sigma^2}{2 b } .\notag
	\end{equation}
	Thus, combining the above two inequalities, we get
	\begin{eqnarray}
		\delta_{t+1} - \delta_t + \fr{1}{2\eta} \fr{b^{\prime}} { (1-p) \cL^2} \rb{G_{t+1} - G_{t}} \leq  -  \eta \mu \delta_t^{\fr{2}{\al}} - \rb{ \fr{1}{2\eta}   \fr{p b^{\prime} } { (1-p) \cL^2}   - \eta  } G_t  + \fr{1}{2\eta} \fr{b^{\prime}} { (1-p) \cL^2} \fr{p\sigma^2}{2b } . \notag
	\end{eqnarray}
	Let $\Psi_t \eqdef \delta_t + \lambda G_t$, $\lambda \eqdef \fr{b^{\prime} } { 2\eta (1-p) \cL^2} $. 
	Using the assumption on the step-size, $\eta \leq \sqrt{\fr{p b^{\prime} }{4(1-p)\cL^2}}$,  we get
	\begin{eqnarray}
		\Psi_{t+1} - \Psi_t &=& \delta_{t+1} - \delta_t +\lambda \rb{G_{t+1} - G_{t}} \notag \\
		 &=& \delta_{t+1} - \delta_t +\fr{b^{\prime} } { 2\eta (1-p) \cL^2} \rb{G_{t+1} - G_{t}} \notag \\
		&\leq&  -  \eta \mu \delta_t^{\fr{2}{\al}} -   \fr{p b^{\prime}} {4\eta  (1-p) \cL^2}   G_t + \fr{1}{2\eta} \fr{b^{\prime}} { (1-p) \cL^2} \fr{p\sigma^2}{2 b }   \notag\\
		&=&  -  \eta \mu \delta_t^{\fr{2}{\al}} -  \fr{p \lambda}{2}  G_t  +  \fr{p \lambda}{2} \fr{\sigma^2}{b}   \notag \\
        &=& -  \eta \mu \rb{\Psi_t - \lambda G_t}^{\fr{2}{\al}}  -  \fr{p \lambda}{2}  G_t  +  \fr{p \lambda}{2} \fr{\sigma^2}{b}  . \notag
	\end{eqnarray}
\end{proof}

\section{Convergence in the Iterates}
In this Section, we assume that $\al$-P{\L} condition holds with $\al \in (1,2]$. We provide convergence guaranties in the \textit{iterates} to the set of optimal points $X^{\star}$, which we assume to be non-empty. The sample complexity results are summarized in Table~\ref{tbl:complexity_x}. The results in Table~\ref{tbl:complexity_x} are obtained by translating the sample complexity results reported in Table~\ref{tbl:complexity_f} to convergence in the iterates via Proposition~\ref{prop:f_to_x}. Note that in the special case $\al = 2$, our rates in both Tables~\ref{tbl:complexity_f} and \ref{tbl:complexity_x} recover the optimal rates for online case \cite{Karimi_PL,Gower_SGD_QC_PL_21,khaled2020better} and the best known results for finite sum case \cite{reddi2016stochastic,PAGE}. \footnote{While our analysis for variance reduction formally holds for $\al < 2$ only, the special case $\al = 2$ can be easily recovered via standard techniques, e.g., \cite{reddi2016stochastic,PAGE}. }

\begin{proposition}\label{prop:f_to_x}
	Let Assumption~\ref{as:lojasiewicz} hold with $\al \in (1, 2]$ and the set of optimal points $X^{\star} \eqdef \argmin_{x}f(x)$ is not empty. Then 
	\begin{eqnarray}\label{eq:f_to_x}
		dist\rb{x, X^\star} \leq \fr{\al}{\al - 1} \fr{1}{\sqrt{2\mu}} \rb{ f(x) - f^{\star} }^{\fr{\al - 1}{\al}} \quad \text{for all } x\in \R^d, 
	\end{eqnarray}
	where $dist\rb{x, X^\star} \eqdef \min_{y\in X^{\star}} \norm{ y - x }$.
\end{proposition}
The above result can be obtained by following the argument similar to the proof of Theorem 2 in \citep{Karimi_PL} (where it is shown for a particular case $\al=2$). The only difference is that one should take a disingularizing function as $g(x) = \rb{f(x) - f^\star}^{\fr{\al - 1}{\al}}$, where $f^\star = \min_{x} f(x)$. This result immediately implies convergence in the iterates via
\begin{eqnarray}
\Exp{ \min_{y\in X^{\star}}  \norm{ x - y}} &=& \Exp{dist\rb{x,X^\star}}  \notag \\
&\overset{\eqref{eq:f_to_x}}{\leq}&  \fr{\al}{\al - 1} \fr{1}{\sqrt{2\mu}} \Exp{ \rb{f(x) - f^\star }^{\fr{\al-1}{\al}} } \notag \\
&\leq&  \fr{\al}{\al - 1} \fr{1}{\sqrt{2\mu}} \rb{\Exp{f(x) - f^\star}}^{\fr{\al-1}{\al}}, \end{eqnarray}
where the last inequality holds by Jensen's inequality for a concave function $ t \mapsto t^{\fr{\al- 1 }{\al}}$.

\begin{table}[h!]
	\caption{Summary of sample complexity results for $\al$-P{\L} functions (Assumption~\ref{as:lojasiewicz}) with $\al \in (1,2]$ under average $\cL$-smoothness (Assumptions~\ref{as:avg_smoothness_page}) and bounded variance (Assumptions~\ref{as:UBV}). Quantities: $\al$ = PL power; $\mu$ = PL constant; $\kappa = \nfr{\cL}{\mu}$; $\sigma^2$ = variance. The entries of the table show the expected number of stochastic gradient calls to achieve  $\Exp{ dist\rb{x, X^\star} } \leq \epsilon_x$, where $X^{\star} \neq \emptyset$ is the set of optimal points of $f(\cdot)$.\vspace{.1cm} }
	\label{tbl:complexity_x}
	\footnotesize
	\centering
	\begin{tabular}{|c|c|c|c|}
		\hline
		\bf Method & \bf Finite sum case & \bf Online case \\
		\hline
		\begin{tabular}{c}\algname{GD} \end{tabular}  & $\cO\rb{ {n \kappa} {\mu^{\fr{\al - 2}{2(\al - 1)}}} \rb{\fr{1}{\epsilon_x}}^{\fr{2-\al}{\al-1}} }$ & N/A \\
		\hline
		\begin{tabular}{c}\algname{SGD} \end{tabular}  & $ \cO\rb{ \kappa \sigma^2 \mu^{\fr{\al + 2}{2(1-\al)}} \rb{\fr{1}{\epsilon_x}}^{\fr{4-\al}{\al - 1}} }$ & $ \cO\rb{ \kappa \sigma^2 \mu^{\fr{\al + 2}{2(1-\al)}} \rb{\fr{1}{\epsilon_x}}^{\fr{4-\al}{\al - 1}} }$ \\
		\hline
		\rowcolor{LightCyan}
		\begin{tabular}{c}\algname{PAGER} \end{tabular}  & $ \cwO\rb{ n + \sqrt{n}\kappa  \mu^{\fr{\al - 2}{2(\al-1)}} \rb{\fr{1}{\epsilon_x}}^{\fr{2-\al}{\al - 1}} }$ (new)  & $ \cO\rb{ \rb{ \fr{\sigma^2}{\mu} + \kappa^2 } \mu^\fr{1}{1 - \al } \rb{\fr{1}{\epsilon_f}}^{\fr{2}{\al - 1}} }$ (new) \\
	\hline
\end{tabular}
\end{table}

\section{Simulations}
In this section, we perform numerical tests to evaluate the performance of the discussed methods. Our experiments are based on the RL setup described in Example~\ref{ex:PO_RL} since we believe that it is one of the most interesting applications of our theoretical results. The goal of our experiments is twofold. First, we want to make sure that variance reduction technique is useful in maximizing a cumulative reward for policy optimization tasks. Second, it is interesting to find out if the restarting procedure in \algname{PAGER} is helpful in practice. 
\paragraph{Algorithmic adjustments.} In order to make Algorithms~\ref{alg:sgd} and \ref{alg:PAGE_w_restarts} applicable to the setup of Example~\ref{ex:PO_RL}, one needs to make some standard adjustments. First, we should specify the way the stochastic gradient is computed. In our experiments, we use the standard GPOMDP estimator \citep{Baxter_GPOMDP}, which is given by
$$
g_k(\theta, \tau) := \fr{1}{b_k}\sum_{i=1}^{b_k}\sum_{h = 0}^{H-1} \gamma^h r(s_h^i, a_h^i) Z_{\theta, h},
$$
where $Z_{\theta, h} := \sum_{z=0}^{h} \nabla_\theta \log \pi_\theta (a_z^i | s_z^i)$, $\tau := \cb{(s_h^i, a_h^i)}_{h=0}^{H-1}$ is generated according to the trajectory distribution $p(\tau | \pi_\theta)$,  $\pi_\theta$ is the parametric policy and $H$ is the horizon length of an episode. Second, the data distribution changes over iterations (distribution shift), and one needs to use an importance weighting technique in order to apply variance reduction methods \citep{papini18a}. Importance weighting is implemented as
$$
g_{k, \omega_{\theta_{2}}}^{\prime}(\theta_{1}, \tau) :=  \fr{1}{b_k^{\prime}}\sum_{i=1}^{b_k^{\prime}} \omega(\tau_i | \theta_2, \theta_1) \sum_{h = 0}^{H-1} \gamma^h r(s_h^i, a_h^i) Z_{\theta, h} \qquad \omega(\tau_i | \theta_2, \theta_1) : = \Pi_{j=0}^{H-1} \fr{\pi_{\theta_1}(a_j^i | s_j^i)}{\pi_{\theta_2}(a_j^i | s_j^i)}.
$$
Given the above notation \algname{PAGE} gradient estimator can be computed as
$$g_{t+1} = \begin{cases}
     g_k(\theta_{t+1}, \tau_{t+1}) ,  & \text{w.p.} \quad\  p,\\
        g_{t} +  g_k^{\prime}(\theta_{t+1}, \tau_{t+1}) - g_{k, \omega_{\theta_{t+1}}}^{\prime}(\theta_{t}, \tau_{t})   , &\text{w.p.}\  1-p.
	\end{cases}
$$
\paragraph{Experimental setup.} We test the discussed methods on benchmark RL environments CartPole and Acrobot that are available on OpenAI gym \citep{brockman2016openai}. Both environments have discrete action space and continuous state space. We use a neural network with two hidden layers of width 32 each and Tanh activation function. We set parameters by default as $H=200$, $\gamma = 0.9999$ and initialize all runs with the same randomly generated policy. For \algname{SGD}, we use $T=1$, $b=50$. For \algname{PAGE} we use $b = 50$, $b^{\prime} = 5$, $p = 0.1$. For \algname{PAGER}, we set initial batch-sizes as $b_0^{\prime} = 15$, $b_0 = 5$, $p_0 = 1$ $T_0 = 50$ and change the values from one stage to another based the formulas given by Theorem~\ref{thm:PAGE_online_w_restarts1} (with $\alpha = 1$). We tune step-sizes from the set $\cb{10^{-5}, 2 \cdot 10^{-5}, \ldots 2^{6} \cdot 10^{-5}}$ and select the one that gives the best performance based on the average reward in the last $10$ iterations. The convergence curves Figure~\ref{fig:RL_exp} are calculated as the mean over multiple runs with fixed parameters, the shaded regions represent one standard deviation.  
\paragraph{Results.} The empirical results shown in Figure~\ref{fig:RL_exp} seem to be in line with our theoretical findings (Theorem~\ref{thm:PAGE_online_w_restarts1}). There are two interesting observations. First, \algname{SGD} requires more time to converge compared to variance reduced methods. The difference is especially tangible for CartPole environment, where \algname{PAGER} stabilizes at the maximal average reward $3$ \textit{times faster} than \algname{SGD}. This is in line with the theoretical sample complexity gap between \algname{PAGER} -- $\cO(\epsilon_f^{-2})$ and \algname{SGD} -- $\cO(\epsilon_f^{-3})$. Second, \algname{PAGER} converges much faster than its (non-restarted) variant \algname{PAGE} on CartPole task, which shows empirically \textit{the benefit of the restarting procedure}. Moreover, the behavior of \algname{PAGER} is \textit{more stable} near optimum. This observation is in accordance with the intuition described in Section~\ref{sec:variace_reduction_appendix} and our theoretical analysis because \algname{PAGER} is able to reduce the variance term in \eqref{eq:simple_recursion_appendix_2} at the desired rate by varying parameters $p$ and $b$ over time.

\begin{figure}[t]
        \centering
        \includegraphics[width=\textwidth]{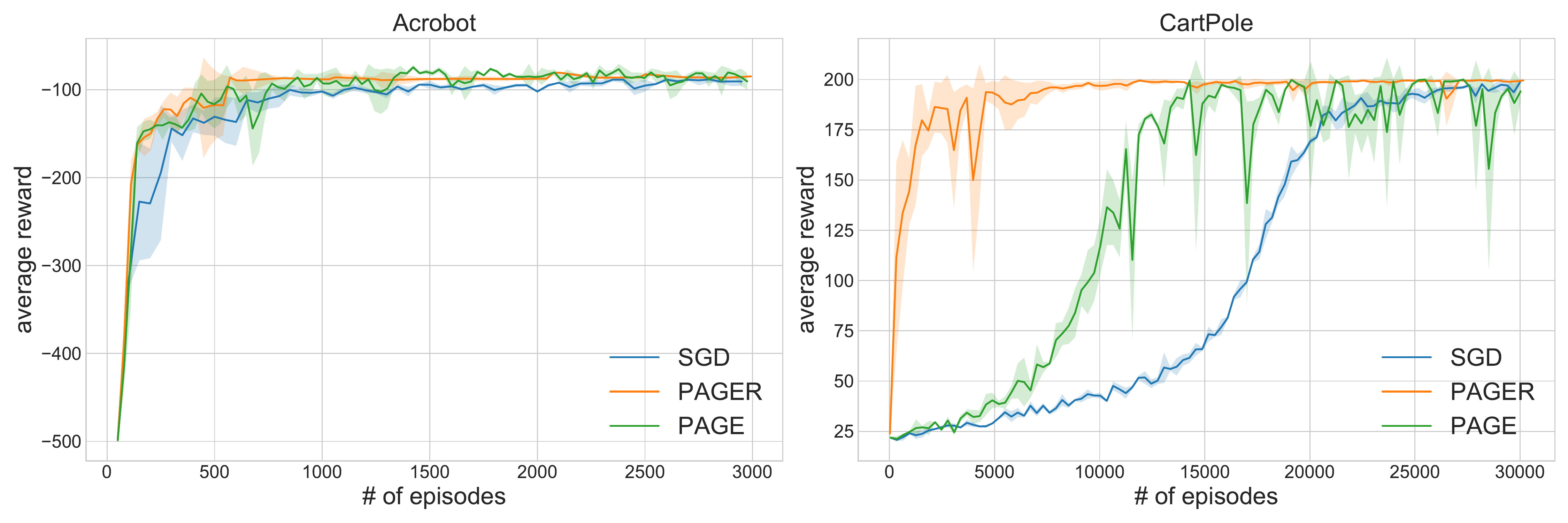}
        \caption{Performance of $\algname{SGD}$, $\algname{PAGER}$ and $\algname{PAGE}$ on benchmark RL tasks.}
        \label{fig:RL_exp}
\end{figure}

\end{document}